\documentclass[11pt,a4paper]{amsart}
\usepackage[utf8]{inputenc}
\usepackage[english]{babel}
\usepackage{amsmath,amssymb,amsthm,amscd,graphicx,mathrsfs,url,dsfont,enumitem}
\usepackage[usenames,dvipsnames]{color}
\usepackage[colorlinks=true,linkcolor=Blue,citecolor=Green,urlcolor=Blue]{hyperref}
\usepackage{dsfont}
\usepackage[super]{nth}
\usepackage[open, openlevel=2, depth=3, atend]{bookmark}
\hypersetup{pdfstartview=XYZ}
\usepackage{epigraph}
\usepackage{mathtools}
\usepackage{enumitem}
\usepackage{mathrsfs}
\usepackage{comment}

\let\oldtocsection=\tocsection
\let\oldtocsubsection=\tocsubsection
\let\oldtocsubsubsection=\tocsubsubsection
\renewcommand{\tocsection}[2]{\hspace{0em}\oldtocsection{#1}{#2}}
\renewcommand{\tocsubsection}[2]{\hspace{1em}\oldtocsubsection{#1}{#2}}
\renewcommand{\tocsubsubsection}[2]{\hspace{2em}\oldtocsubsubsection{#1}{#2}}

\def\?[#1]{\textbf{[#1]}\marginpar{\Large{\textbf{??}}}}

\usepackage{import}
\usepackage{xifthen}
\usepackage{pdfpages}
\usepackage{transparent}
\usepackage{chngcntr}
\counterwithin{figure}{section}
\usepackage{subcaption}

\numberwithin{equation}{section}

\newtheorem{theorem}{Theorem}[section]
\newtheorem{lemma}[theorem]{Lemma}

\newtheorem{proposition}[theorem]{Proposition}

\newtheorem{conjecture}[theorem]{Conjecture}
\theoremstyle{remark}
\newtheorem{remark}{Remark}
\newtheorem{example}{Example}

\theoremstyle{definition}
\newtheorem{definition}[theorem]{Definition}

\newcommand{\kg}{\mathfrak g}
\newcommand{\C}{\mathbb C}
\renewcommand{\S}{\mathbb S}
\newcommand{\D}{\mathbb D}
\newcommand{\cA}{\mathcal A}
\renewcommand{\d}{\mathrm d}
\newcommand{\kB}{\mathfrak B}
\newcommand{\del}{\partial}
\newcommand{\cC}{\mathcal C}
\newcommand{\cJ}{\mathcal J}
\newcommand{\R}{\mathbb R}
\newcommand{\cL}{\mathcal L}
\newcommand{\E}{\mathbb E}
\newcommand{\delbar}{\bar\partial}
\newcommand{\Z}{\mathbb Z}
\renewcommand{\P}{\mathbb P}
\newcommand{\km}{\mathrm{KM}}
\renewcommand{\Re}{\mathrm{Re}}
\renewcommand{\Im}{\mathrm{Im}}
\newcommand{\bS}{\mathbf S}
\newcommand{\ind}{\mathbf 1}
\newcommand{\cM}{\mathcal M}
\newcommand{\cQ}{\mathcal Q}
\newcommand{\bJ}{\mathbf{J}}
\newcommand{\Ad}{\mathrm{Ad}}
\newcommand{\cT}{\mathcal{T}}
\newcommand{\bk}{\mathbf{k}}
\newcommand{\bL}{\mathbf{L}}
\newcommand{\bR}{\mathbf{R}}
\newcommand{\N}{\mathbb{N}}
\newcommand{\cU}{\mathcal{U}}
\newcommand{\cV}{\mathcal{V}}
\newcommand{\cW}{\mathcal{W}}
\newcommand{\cR}{\mathcal{R}}
\newcommand{\cH}{\mathcal{H}}
\newcommand{\ralpha}{\text{\reflectbox{$\alpha$}}}

\newcommand{\tr}{\mathrm{tr}}
\newcommand{\sS}{\mathscr{S}}
\newcommand{\sE}{\mathscr{E}}

\title{Unitarising measures for Kac--Moody algebras}
\author{Guillaume Baverez}
\email{guillaume.baverez@bicmr.pku.edu.cn}
\address{Beijing International Center for Mathematical Research, Peking University}
\date{}


\begin{document}

\begin{abstract}
Given a compact connected Lie group $G$ with dual Coxeter number $\check h$ and a level $\kappa<-2\check h$, we introduce a probability measure $\nu_\kappa$ on the space of holomorphic $\kg_\C$-valued $(1,0)$-forms in $\D$, in relation to the K\"ahler geometry of the loop group of $G$ and the action of a pair of Kac--Moody algebras at respective levels $\kappa$ and $-2\check h-\kappa>0$. We prove that $\nu_\kappa$ is characterised by a covariance property making rigorous sense of the formal path integral ``$\d\nu_\kappa(\gamma)=e^{-\check\kappa\sS(\gamma)}D\gamma$", where $D\gamma$ is the non-existent Haar measure on the loop group and $\sS$ is a K\"ahler potential for the right-invariant Kac--Moody metric. Infinitesimally, the covariance formula prescribes the Shapovalov forms of the Kac--Moody representations.

\end{abstract}

\maketitle


\tableofcontents

\section{Introduction}


    \subsection{Integration on loop groups and the Langlands correspondence}\label{subsec:langlands}
The main references for this section are \cite{Frenkel07_CFT,Frenkel20_is-there,EFK1,EFK2,Langlands}.
    
The \emph{arithmetic} Langlands programme is a vast web of conjectures relating elliptic curves to modular forms, and can be understood at first order as the harmonic analysis of $\mathrm{Bun}_{G_\mathbb{F}}(X)$. Here, $\mathbb F$ is a finite field, $G_\mathbb F$ is a connected reductive algebraic group over $\mathbb F$, $X$ is a smooth projective curve over $\mathbb F$, and $\mathrm{Bun_{G_\mathbb F}}(X)$ is the moduli stack of $G_\mathbb F$-bundles over $X$. Crucially, $\mathrm{Bun}_{G_\mathbb{F}}(X)$ inherits a canonical measure from the Haar measure on $G_{\mathbb{F}}(\mathbb A_{\mathbb F})$---the group of $\mathbb A_\mathbb F$-points of $G_\mathbb F$, where $\mathbb A_\mathbb F$ is the ad\`ele ring of $\mathbb F$. This allows one to speak of square-integrable functions on $\mathrm{Bun}_{G_\mathbb F}(X)$ (a.k.a. the \emph{automorphic forms}) and to define a family of integral transforms on $L^2(\mathrm{Bun}_{G_\mathbb F}(X))$ known as the \emph{Hecke operators}. The Langlands correspondence seeks to understand the spectral data of these operators (a.k.a. the \emph{Langlands parameters}) in relation to the representation theory of the Langlands dual group $^LG_\mathbb F$. 

The \emph{geometric} Langlands correspondence aims to address the same questions when $\mathbb F$ is replaced by the field of complex numbers $\C$, so that $G_\C$ is now a complex Lie group (more precisely, the complexification of a compact connected Lie group $G$), $X$ is a compact Riemann surface, and the moduli stack $\mathrm{Bun}_{G_\C}(X)$ can be realised as a double quotient of the loop group $LG_\C=C^\infty(\S^1;G_\C)$.\footnote{Other notions of regularity will be discussed below, the smooth topology playing no specific role at this point.} The main difference here is that there is no preferred measure on $\mathrm{Bun}_{G_\C}(X)$ since $LG_\C$ does not carry any non-trivial Haar measure (it is not locally compact), making it difficult to even \emph{state} a spectral problem: it is not clear how to define a Hilbert space, let alone the Hecke operators acting on it. See  \cite[Sections~1-2]{Langlands} and \cite[Sections 1.6 \& 3]{Frenkel20_is-there} for more on this topic, and in particular \cite[Section 1.1]{EFK2} for the appearance of infinite dimensional integration in the definition of Hecke operators. This issue led Beilinson \& Drinfeld to formulate the geometric Langlands correspondence in a \emph{categorical} language \cite{Beilinson-Drinfeld}, where for instance eigenfunctions are traded for eigensheaves. The fulfilment of this programme was one of the major mathematical achievements of the past years~\cite{Gaitsgory}. 

Recently, Langlands himself proposed a new approach to the geometric theory which is closer in spirit to the arithmetic one~\cite{Langlands}, making use of ``traditional" functional analysis and avoiding sophisticated algebro-geometric concepts. 
A related approach was also suggested in the mathematical physics literature \cite{Teschner18_ALC}, with motivations coming from quantum field theory and integrable systems (for recent progress and conjectures, see \cite{Gaiotto-Witten,Gaiotto-Teschner25}). The systematic mathematical study of this \emph{analytic} Langlands correspondence was initiated in \cite{EFK1,EFK2}; here, the Hilbert space is the space of square-integrable half-densities, which is acted on by the algebra of differential operators on $\mathrm{Bun}_{G_\C}(X)$. Due to the unbounded nature of these operators, delicate questions of functional analysis are brought to the forefront. A lot of work is put in the definition of Hecke operators, and the study of the corresponding spectral problem is the subject of many conjectures \cite[Section 1.5]{EFK1}, most of which are still open at the time of writing. The main difference with the Beilinson--Drinfeld approach is that one should restrict to \emph{single-valued}, \emph{square-integrable} eigensections, and the Langlands parameters are conjectured to be in correspondence with the set of \emph{$^LG_\C$-opers} on $X$ having \emph{real monodromy} \cite[Conjecture 1.9]{EFK1}. 

Importantly, the use of half-densities (rather than functions) gives a construction of the Hilbert space which avoids any integration theory on $LG_\C$. The goal of this article is to come back to the root of this question and construct a family of probability measures $(\nu_\kappa)_{\kappa<-2\check h}$ on the space $\cA$ of holomorphic $\kg_\C$-valued $(1,0)$-forms on $\D$  (which can be understood as a loop group orbit), where the number $\kappa$ refers to the level of a certain representation of the Kac--Moody algebra. These measures are characterised by a covariance formula with respect to a natural action of $LG$ on $\cA$, and the infinitesimal version of this formula specifies the Shapovalov form of the Kac--Moody algebras acting on $L^2(\nu_\kappa)$. 

    \subsection{What we learn from \texorpdfstring{$\mathrm{Diff}(\S^1)$}{Diff(S1)} and Schramm--Loewner evolutions}\label{subsec:blabla}
Loop groups are one of the two ubiquitous infinite dimensional Lie groups arising in the mathematics of two-dimensional quantum field theory, the other one being $\mathrm{Diff}(\S^1)$---the group of orientation preserving diffeomorphisms of the unit circle. The latter possesses a rich theory of integration known as \emph{Schramm--Loewner evolutions} (SLE), from which our main result is greatly inspired. This section recalls the main features of this theory and draws the analogies with the problem at hand. 
    
The turn of the millennium saw a great deal of activity on the topic of random Jordan curves. On the one hand, Schramm introduced the stochastic Loewner evolution (now Schramm--Loewner evolution or SLE) as the candidate scaling limit of interfaces of statistical systems at the critical temperature \cite{Schramm00}.\footnote{Here, we are only concerned with the SLE$_\kappa$ loop for $\kappa\in(0,4]$, which is an infinite, M\"obius-invariant measure on Jordan curves for each $\kappa$ \cite{Zhan21}, supported on curves of Hausdorff dimension $1+\frac{\kappa}{8}$.} Concomitantly and independently \cite{AM01}, Airault \& Malliavin formulated an axiomatic description of measures on Jordan curves (dubbed \emph{unitarising measures for the Virasoro algebra}) based on conformal welding and Kirillov's action \cite{Kirillov98}, but they did not succeed in constructing such measures. Motivated by the property of SLE called \emph{conformal restriction} \cite{LSW03,Zhan21}, Kontsevich \& Suhov proposed the far-reaching conjecture that SLE and Airault--Malliavin measures are the same \cite[Section~2.5.2]{KS07}. The proof of this conjecture is contained in \cite{BJ24,BJ25,Baverez25}, with a surprising twist that we sketch below.

 In \cite{BJ24,BJ25}, it is observed that the $L^2$-space of the SLE$_\kappa$ loop measure carries simultaneously \emph{two} pairs of commuting representations of the Virasoro algebra: a non-unitary one (but with a natural Shapovalov form\footnote{The Shapovalov form of a Virasoro representation $(L_n)_{n\in\Z}$ acting on a Hilbert space $\cH$ is defined as a Hermitian form $\cQ$ on $\cH$ such that $L_{-n}$ is the $\cQ$-adjoint of $L_n$ for all $n\in\Z$. A representation is unitary if $\cQ$ coincides with the inner-product on $\cH$.}) with central charge $c_\mathrm{m}=1-6(\frac{2}{\sqrt{\kappa}}-\frac{\sqrt{\kappa}}{2})^2\leq1$, and a unitary one with central charge $c_\mathrm{L}=26-c_\mathrm{m}\geq25$. These two representations express different ways of encoding the local symmetry acting on the space of Jordan curves: via conformal restriction \cite{LSW03,Zhan21} or via Kirillov's action \cite{Kirillov98,AM01}. The discrepancy between the two central charges can be understood as the Jacobian of the conformal welding map, giving the change of basis between the two descriptions of the tangent bundle \cite{BJ25}. The outcome is that one can heuristically understand the SLE$_\kappa$ loop measure as the Feynman path integral $e^{-\frac{c_\mathrm L}{12\pi}\bS(\eta)}D\eta$, where $D\eta$ is the non-existent Kirillov-invariant measure on the space of Jordan curves and $\bS$ is the K\"ahler potential for the Weil--Petersson metric on the universal Teichm\"uller space, a.k.a the \emph{universal Liouville action}~\cite{Takhtajan-Teo06}.

 In this paper, we address the same question for loop groups rather than $\mathrm{Diff}(\S^1)$, i.e. we construct a measure making rigorous sense of the path integral ``$\d\nu_\kappa(\gamma)=e^{-\check\kappa\sS(\gamma)}D\gamma$" for all $\check\kappa=-2\check h-\kappa>0$, where $D\gamma$ is the non-existent Haar measure on $LG$ and $\sS$ is a K\"ahler potential for the right-invariant Kac--Moody metric on (a subgroup of) $LG$, which we call the \emph{universal Kac--Moody action} (see Section \ref{subsec:potential}). As opposed to the $\mathrm{Diff}(\S^1)$-case, $\sS$ is not invariant under $\gamma\mapsto\gamma^{-1}$, leading to the non-reversibility of $\nu_\kappa$ (contrary to the SLE loop measure). In fact, the very definition of the law of the inverse is non-obvious since $\nu_\kappa$ does not charge the space of continuous loops. See Section \ref{subsec:shapo} for more on this issue. 
 
 Involved in our construction are two pairs of commuting representations of the Kac--Moody algebra at levels $\kappa$ and $-2\check h-\kappa$ respectively, encoding two distinct realisations of the local symmetry of the model. The number $-2\check h$ plays the same role as 26 in the case of $\mathrm{Diff}(\S^1)$, expressing the Jacobian of the change of coordinate between real loops and holomorphic forms (a.k.a. the Birkhoff factorisation, see Section \ref{subsec:factorisation}). Compared to \cite{BJ24,BJ25}, the main technical difficulty is the absence of a candidate measure for $\nu_\kappa$, and the existence and uniqueness of this measure is the main contribution of this work.

 Finally, we note that the condition $\kappa<-2\check h$ strictly excludes the \emph{critical level} $\kappa=-\check h$, which is usually the one considered in the Langlands correspondence, due to the degeneracy of the Sugawara operators to a commuting family and their identification with the quantum Hitchin Hamiltonians \cite[Section 8]{Frenkel07_CFT}. The non-critical $\kappa$-values are sometimes called the \emph{quantum} analytic Langlands correspondence in mathematical physics \cite{Gaiotto-Teschner25}.
 
    \subsection{Main result}\label{subsec:results}
 In the spirit of not overloading the introduction with technicalities, we will simply sketch some parts of the setup and refer to the subsequent sections for details. 

Fix a compact connected Lie group $G$ with complexification $G_\C$ and respective Lie algebras $\kg,\kg_\C$, and let $L^\omega_1G$ be the subgroup of the loop group of $G$ exhibited in Section \ref{subsec:factorisation} (see in particular Lemma \ref{lem:subgroup}). Let $\hat D^\omega_\infty G_\C$ be the group of holomorphic functions $g:\D^*\to G_\C$ such that $g(\infty)=1_{G_\C}$. Let $\cA$ (resp. $\cA^\omega$) be the space of holomorphic $\kg_\C$-valued $(1,0)$-forms in $\D$ (resp. a neighbourhood of $\bar\D$), and endow $\cA$ with the local uniform topology (making it a Polish space).

In Section \ref{subsec:potential}, we define a function $\sS:\cA^\omega\to\R_+$ which is a K\"ahler potential for the right-invariant Kac--Moody metric (defined in Section \ref{subsec:loop_groups}). In Section \ref{subsec:real_action}, we define a left action of $L^\omega_1G$ on $\cA$, denoted $(\chi,\alpha)\mapsto\chi\cdot\alpha$ for $\chi\in L^\omega_1G$, $\alpha\in\cA$. While $\sS$ is not defined on the whole of $\cA$, Proposition \ref{prop:Omega} states that there is $\Omega\in C^0(L^\omega_1G\times\cA)$ extending the values of the function $(\chi,\alpha)\mapsto\sS(\chi\cdot\alpha)-\sS(\alpha)$ on $L^\omega_1G\times\cA^\omega$. Similarly, we define in Section \ref{subsec:action_complex} a right action of $\hat D^\omega_\infty G_\C$ on $\cA$, denoted $(\alpha,h)\mapsto\alpha\cdot h$, and there is a function $\Lambda\in C^0(\cA\times\hat D^\omega_\infty G_\C)$ extending the values of $(\alpha,h)\mapsto\sS(\alpha\cdot h)-\sS(\alpha)$ on $\cA^\omega\times\hat D^\omega_\infty G_\C$.

\begin{theorem}\label{thm:main}
    Fix $\check\kappa=-2\check h-\kappa>0$. 
    \begin{enumerate}[label=\arabic*.]
\item There exists a unique Borel probability measure $\nu_\kappa$ on $\cA$ such that for all $\chi\in L^\omega_1G$ and all $F\in C^0_b(\cA)$,
\begin{equation}\label{eq:quasi-invariance}
\int_\cA F(\chi^{-1}\cdot\alpha)\d\nu_\kappa(\alpha)=\int_\cA F(\alpha)e^{-\check\kappa\Omega(\chi,\alpha)}\d\nu_\kappa(\alpha).
\end{equation}
\item The measure $\nu_\kappa$ is the unique Borel probability measure on $\cA$ such that for all $h\in\hat D^\omega_\infty G_\C$ and all $F\in C^0_b(\cA)$,
\begin{equation}\label{eq:quasi-invariance-bis}
    \int_\cA F(\alpha\cdot h^{-1})\d\nu_\kappa(\alpha)=\int_\cA F(\alpha)e^{\kappa\Lambda(\alpha,h)}\d\nu_\kappa(\alpha)
\end{equation}
\end{enumerate}
We call $\nu_\kappa$ the \emph{unitarising measure for the Kac--Moody algebra at level $\check\kappa$}.
\end{theorem}

A few comments are in order. Even though the Birkhoff factorisation allows us to identify $\cA^\omega$ with an orbit of the analytic loop group, the measure $\nu_\kappa$ gives 0 mass to the set of those $\alpha\in\cA$ admitting a continuous extension to $\bar\D$, so there is no way to associate a loop in $G$ to a sample of $\nu_\kappa$. In particular, $\sS(\alpha)=\infty$ for $\nu_\kappa$-a.e. $\alpha\in\cA$. Still, it is instructive to think heuristically that
\begin{equation}\label{eq:nu_heuristic}
``\,\d\nu_\kappa(\gamma)=e^{-\check\kappa\sS(\gamma)}D\gamma\,",
\end{equation}
where $D\gamma$ is the (non-existent) Haar measure on the loop group: indeed, the transformation rule prescribed by \eqref{eq:quasi-invariance} is precisely the one that would be satisfied by the right-hand-side of \eqref{eq:nu_heuristic}.

One of the key aspects of our proof is to understand the change of coordinate between the loop group and the space of holomorphic forms $\cA$ (equivalently, the space $D_0G_\C$ of holomorphic functions $g:\D\to G_\C$ with $g(0)=1_{G_\C}$), given by the Birkhoff factorisation. This gives another interpretation of $\nu_\kappa$ as
\[``\,\d\nu_\kappa(g)=e^{-\check\kappa\sS(g)-2\check h\hat\sS(g)}Dg\,"\qquad\text{or}\qquad``\,\d\nu_\kappa(\alpha)=e^{-\check\kappa\sS(\alpha)-2\check h\hat\sS(\alpha)}D\alpha\,",\]
where $Dg$ (resp. $D\alpha$) is the (non-existent) Haar (resp. Lebesgue) measure on $D_0G_\C$ (resp. $\cA$), and we have set $\hat\sS(\gamma):=\sS(\gamma^{-1})$ for all $\gamma\in L^\omega_1G$. In other words, we can think of $e^{-2\check h\hat\sS}$ as the Jacobian of the Birkhoff factorisation map. Note that there is no obvious way to extend $\hat\sS$ to a function defined $\nu_\kappa$-a.e. since its samples are not continuous loops, so this is a purely formal interpretation (the rigorous sense of which is the content of Lemma \ref{lem:induction}).

Infinitesimally, the transformation properties of $\nu_\kappa$ have an important counterpart with respect to certain representations of the Kac--Moody algebra introduced in Sections \ref{sec:km_complex} and~\ref{sec:km_real}. Namely, by differentiating \eqref{eq:quasi-invariance} and \eqref{eq:quasi-invariance-bis} at $\chi=\ind_G$ and $h=\ind_{G_\C}$ respectively, one gets an expression for their Shapovalov forms. In particular, Theorem \ref{thm:unitary} motivates the terminology ``unitarising measure" by analogy with \cite{AM01}. Since the statement is a bit technical, we postpone it to Section~\ref{subsec:shapo}.

\begin{remark}
    In the case $G=\S^1$ ($\kg_\C=\C$), the Kac--Moody algebra reduces to a Heisenberg algebra, for which unitarising measures are well-known to be Gaussian \cite[Section 2]{Malliavin-unitarizing}. The reader can check in this case that indeed $\sS(\alpha)=\frac{i}{4\pi}\int_\D\alpha\wedge\bar\alpha$, and the covariance formulas in Theorem \ref{thm:main} are just the usual Cameron--Martin shifts. The random differential form is better described as $\alpha=\frac{i}{\sqrt{\check\kappa}}\del\varphi$ where $\varphi:\D\to\R$ is the Gaussian harmonic function with covariance kernel $\E[\varphi(z)\varphi(\zeta)]=\log|1-z\bar\zeta|$. In the sequel, we will only focus on the non-trivial case where $G$ is non-Abelian.
\end{remark}



    \subsection{Discussion and outlook}

        \subsubsection{Relationship with existing literature}\label{subsec:Wiener}
To the best of our knowledge, all the research in probability theory concerning measures on loop groups has focused on the Wiener measure, for which there is an extensive literature starting from the Malliavins \cite{Malliavins}. We refer to \cite{Driver_proceedings} for a review and additional references. Additionally, Pickrell proved that the Wiener measure has a weak limit in the large temperature limit \cite{Pickrell00,Pickrell06}, which should also correspond to the $\kappa\nearrow-2\check h$ (or $\check\kappa\searrow0$) limit of the unitarising measure $\nu_\kappa$. See also \cite[Section 3.2]{Frenkel20_is-there} for a discussion of more exotic integration theories on loop spaces (motivic, power-series valued...) and what appears to be their limitations in the context of the analytic Langlands correspondence.

There is a number of quick ways to see that the unitarising measure $\nu_\kappa$ and the Wiener measure are mutually singular. First, samples of the Wiener measure are almost surely continuous loops, while $\nu_\kappa$ does not charge this set. This is the reason why we must define $\nu_\kappa$ on the space of $\kg_\C$-valued holomorphic forms in $\D$, since samples do not converge on the boundary. Second, the action functional for the Wiener measure is the kinetic energy 
\begin{equation}\label{eq:kinetic}
\sE(\gamma):=\frac{1}{2i\pi}\oint_{\S^1}\tr(\gamma^{-1}\del_z\gamma,\gamma^{-1}\del_z\gamma)z\d z\geq0,\qquad\gamma\in L^\omega G,
\end{equation}
while that for $\nu_\kappa$ is the potential $\sS$ introduced in Section \ref{subsec:potential}. These two actions are quite different, as seen from their expansion near $\gamma=\mathbf 1_G$: writing $\gamma_t=e^{tu+\bar tu^*}$ for some $u\in D_0^\omega\kg_\C$ and $u^*(z)=\overline{u(1/\bar z)}$, the kinetic energy behaves like $|t|^2\oint\tr(\del_zu,\del_zu^*)\frac{z\d z}{2i\pi}$ as $t\to0$. On the other hand, since $\sS$ is a K\"ahler potential for the Kac--Moody metric on $LG/G$, we have $\sS(\gamma_t)\sim\frac{|t|^2}{4i\pi}\oint_{\S^1}\tr(u^*\del u)$ as $t\to0$. So the kinetic energy behaves like a Sobolev $H^1$-norm on $\S^1$, while $\sS$ behaves like an $H^{1/2}$-norm. These two norms induce mutually singular Gaussian measures in the Lie algebra of the loop group: the first is the $\kg$-valued Brownian bridge, while the second is the $\kg$-valued $\log$-correlated field on $\S^1$. Intuitively, one can think of $\nu_\kappa$ as a ``$G$-valued $\log$-correlated field on $\S^1$", even though this does not make literal sense.


        \subsubsection{Coupling \texorpdfstring{$\nu_\kappa$}{unitarising measures} with the \texorpdfstring{$G_\C/G$-}{coset }WZW model}
Let $\kappa<-2\check h$. On the one hand, one can interpret $\nu_\kappa$ as a random $G_\C$-bundle over the Riemann sphere (trivialised in the charts $\D$ and $\D^*$). On the other hand, given a fixed (stable) $G_\C$-bundle on a Riemann surface, the $G_\C/G$-Wess--Zumino--Witten model can be interpreted as random metric in this bundle \cite{Gawedzki-Kupiainen,Gawedzki91}. In the spirit of the coupling between Liouville CFT and Schramm--Loewner evolutions \cite{Sheffield16,AHS23,BJ25}, we plan to investigate the coupling between our measure $\nu_\kappa$ with these WZW models, starting with the case $G=\mathrm{SU}_2$ (which is the only case for which a mathematical construction of the path integral is known \cite{GKR25_H3}) on the Riemann sphere.

        \subsubsection{Measures on \texorpdfstring{$\mathrm{Bun}_{G_\C}$}{moduli spaces} and the Langlands correspondence}
Given a compact Riemann surface $X$, the \emph{Kac--Moody uniformisation} describes the moduli space  $\mathrm{Bun}_{G_\C}(X)$ as the double quotient $G_\C(X\setminus\{p\})\big\backslash L^\omega G_\C\big/D^\omega G_\C$ \cite[Chapter~18.1.6]{Frenkel04_book}, where $p\in X$ is an arbitrary point and $G_\C(X\setminus\{p\})$ denotes the space of meromorphic $G_\C$-valued functions on $X$ with poles only allowed at $p$. The isomorphism is obtained by describing a bundle over two local chart (a neighbourhood of $p$ and $X\setminus\{p\}$), equivalently a transition function on the overlap (defining a point in $L^\omega G_\C$). This suggests a way to endow $\mathrm{Bun}_{G_\C}(X)$ with a measure by passing our measure $\nu_\kappa$ to the quotient by integrating in the fibres of the projection. 

A \emph{Hecke modification} consists in changing this local data defining the bundle, and the Hecke operators are integral transforms over the space of such modifications \cite{EFK2} which typically looks like the space of $G_\C$-valued functions on $X$ defined away from a neighbourhood of $p$. The measures introduced in this paper could give another approach to these integral transforms. See Section \ref{subsec:shapo} for some suggestions.

        \subsubsection{Sugawara construction and Drinfeld--Sokolov reduction}
The Sugawara construction is an embedding of the Virasoro algebra into the Kac--Moody algebra. A natural step forward will be to construct these operators in the functional analytic language, acting as unbounded operators on $L^2(\nu_\kappa)$. It would also be interesting to see if this representation generates a family of Markov processes with values in $\cA$, analogously to the results of \cite{BGKRV24}. More generally, the space $L^2(\nu_\kappa)$ should carry a representation of the full $\cW$-algebra of $\kg_\C$, as suggested by the results of quantum Drinfeld--Sokolov reduction \cite[Chapter 16.7]{Frenkel04_book}.

    \subsection{Index of notations}
The right panel shows the first place it is introduced in the main text. 

$G,G_\C$: A compact connected Lie group, its complexification. \hfill Section \ref{subsec:lie_groups}


$\check h$: Dual Coxeter number. \hfill \eqref{eq:coxeter}

$\tr$: Killing form, normalised so that the longest root has squared length 2.\hfill Section \ref{subsec:lie_groups}


$\kB\subset i\kg$: Orthonormal basis of $\kg_\C$ (over $\C$) with respect to $\tr$.\hfill Section \ref{subsec:lie_groups}

$L^\omega G$: Analytic loop group of $G$.\hfill Section \ref{subsec:loop_groups}

$L^\omega\kg,L^\omega\kg_\C$: Analytic loop algebra, its complexification. Canonical generators $b_m=b\otimes z^m$ for $b\in\kB,m\in\Z$. Subspace $L^\omega_0\kg_\C$ of mean-zero loops.\hfill Section \ref{subsec:loop_groups}, \eqref{eq:loop-algebra}

$\kg_\C(z)$: Laurent polynomials with coefficients in $\kg_\C$. 

$D_0G_\C$ (resp. $\hat D_\infty G_\C)$: Holomorphic functions $g:\D\to G_\C$ (resp. $g:\D^*\to G_\C$) with $g(0)=1_G$ (resp. $g(\infty)=1_G$).\hfill Section \ref{subsec:factorisation}

$D_0\kg_\C,\hat D_\infty\kg_\C$: Lie algebras of $D_0G_\C$, $\hat D_\infty G_\C$.\hfill Section \ref{subsec:factorisation}

$\cA$: Holomorphic $\kg_\C$-valued $(1,0)$-forms in $\D$. \hfill Section \ref{subsec:results}

$\omega_\km$: Kac--Moody cocycle on $L^\omega_0\kg_\C$. \hfill \eqref{eq:def-omega}


$\cC=\cC^{1,0}\otimes\cC^{0,1}$: Polynomials on $\cA$.\hfill \eqref{eq:def_polynomials}


$(\cJ_u,\bar\cJ_u)_{u\in L^\omega\kg_\C}$: Complex representations of the loop algebra.\hfill Definition \ref{def:diff-complex}, Lemma \ref{lem:commute-cJ}

$(\bJ_u,\bar\bJ_u)_{u\in L^\omega\kg_\C}$: Kac--Moody modification at level $\kappa<-2\check h$.\hfill \eqref{eq:def-bJ}, Proposition \ref{prop:km_complex}

$\Psi_{\bk,\tilde\bk}$: States in the highest-weight representation generated by $(\bJ_u,\bar\bJ_u)_{u\in L^\omega\kg_\C}$.\hfill \eqref{eq:def_W}

$(\cL_u)_{u\in L^\omega\kg_\C}$: Real representation of the loop algebra. \hfill Definition \ref{def:test_function}, Lemma \ref{lem:commute-cL}

$(\bL_u)_{u\in L^\omega\kg_\C}$: Kac--Moody modification at level $\check\kappa=-2\check h-\kappa>0$.\hfill \eqref{eq:def-bL}, Proposition \ref{prop:km_real}

$\sS$: K\"ahler potential for the right-invariant Kac--Moody metric.\hfill Proposition \ref{prop:potential}

$\Omega,\Lambda$: Covariance moduli.\hfill Proposition \ref{prop:Omega}

$\P(\cA)$: Borel probability measures on $\cA$.\hfill\eqref{eq:Borel}




\medskip

Typically, we denote an element of $L^\omega G$ by $\gamma$, an element of $D_0G_\C$ by $g$, and an element of $L^\omega G_\C$ by $h$. If $g\in D_0G_\C$, we usually write $\alpha=g^{-1}\del g$ and $\ralpha=-\del gg^{-1}$.

An $\omega$-superscript on a space of holomorphic functions (or forms) means ``holomorphic extension to a neighbourhood of the closure of the domain of definition" (e.g. $\cA^\omega$ is the space of analytic $\kg_\C$-valued $(1,0)$-forms in the neighbourhood of $\bar\D$). 

The dual space to $\cA$ is naturally identified with $\hat D_\infty^\omega\kg_\C$ via the pairing $\alpha(u)=\frac{1}{2i\pi}\oint\tr(u\alpha)$, where the contour is the circle $e^{-\delta}\S^1$ for some arbitrarily small $\delta>0$.

The $*$-superscript denotes the pullback by $z\mapsto1/\bar z$ and depends on the type of object on which it acts, e.g. $u^*(z)=\overline{u(1/\bar z)}$ if $u\in L^\omega\kg_\C$ and $\alpha^*(z)=-z^{-2}\overline{\alpha(1/\bar z)}$ if $\alpha\in\cA$.

    \subsection{Outline}

In Section \ref{sec:preliminaries}, we record some elementary material on Lie groups, their loop groups, and the Birkhoff factorisation.

In Section \ref{sec:km_complex}, we introduce two commuting representations $(\bJ_u,\bar\bJ_u)_{u\in L^\omega\kg_\C}$ of the Kac--Moody algebra at level $\kappa<-2\check h$, acting as endomorphisms of a suitable space of polynomials on $\cA$. These representations are used to build a family of orthogonal polynomials for the measure $\nu_\kappa$, and encode the infinitesimal version of the covariance formula \eqref{eq:quasi-invariance-bis}.

In Section \ref{sec:km_real}, we introduce another pair of commuting representations of the Kac--Moody algebra at dual level $\check\kappa=-2\check h-\kappa>0$. The representation $(\bL_u)_{u\in L^\omega\kg\C}$ encodes the infinitesimal version of the covariance formula \eqref{eq:quasi-invariance}. Section \ref{subsec:potential} introduces the K\"ahler potential $\sS$ and establishes a few basic properties, among which its non-negativity. It also introduces the covariance moduli $\Omega$ and $\Lambda$.

The actual proof of Theorem \ref{thm:main} is done in Section \ref{sec:proof}, with Section \ref{subsec:existence} addressing the existence, Section \ref{subsec:mgf} addressing uniqueness, and Section \ref{subsec:uniqueness} exhibiting the link between the two covariance formulas in Theorem \ref{thm:main}. The remaining Section \ref{subsec:shapo} presents a partly conjectural description of the Shapovalov form of the representation $(\bJ_u)_{u\in L^\omega\kg_\C}$.

    \section{Preliminaries}\label{sec:preliminaries}

    
        \subsection{Lie groups}\label{subsec:lie_groups}
In this section, we collect some standard material from Lie group and Lie algebraic theory. The main reference is \cite[Chapter 2.4.1]{Frenkel04_book}.
        
Let $G$ be a compact connected Lie group, with complexification $G_\C$ and an antiholomorphic involution $g\mapsto\bar g$ on $G_\C$ with fixed set $G$ (see Example \ref{example} for a concrete instance). We denote by $\kg$ and $\kg_\C\simeq\kg\otimes\C$ their respective Lie algebras, and view $\kg$ as a subspace of $\kg_\C$. We have a real form $u\mapsto\bar u$ on $\kg_\C$ with fixed set $\kg$. As a complex Lie algebra, $\kg_\C$ has a $\C$-bilinear Lie bracket $[\cdot,\cdot]:\kg_\C\times\kg_\C\to\kg_\C$, which is skew-symmetric and satisfies the Jacobi identity
\[[x,[y,z]]+[y,[z,x]]+[z,[x,y]]=0,\qquad\forall x,y,z\in\kg_\C.\] 

The group $G_\C$ acts on itself by conjugation $(g,h)\mapsto ghg^{-1}$ for $g,h\in G_\C$. The differential at $h=1_G$ defines an action of $G_\C$ on $\kg_\C$ called the \emph{adjoint representation} and denoted $\Ad:G_\C\to\mathrm{GL}(\kg_\C)$. We will sometimes abuse notations by writing $\Ad_g(x)=gxg^{-1}$ for $g\in G_\C$, $x\in\kg_\C$ (matrix notation). Differentiating again at $g=1_G$ gives an action of $\kg_\C$ on itself (also called the adjoint representation) and given by $\mathrm{ad}:\kg_\C\to\mathrm{End}(\kg_
\C),\,x\mapsto[x,\cdot\,]$.

Every complex representation $\rho:\kg_\C\to\mathrm{End}(V)$ induces a symmetric $\C$-bilinear form on $\kg_\C$ by
\[\kappa_\rho(x,y):=\tr_V(\rho(x)\circ\rho(y)).\]
This form satisfies the invariance properties
\begin{equation}\label{eq:killing}
\kappa_\rho(\Ad_g(x),\Ad_g(y))=\kappa_\rho(x,y)\qquad\text{and}\qquad\kappa_\rho(x,[y,z])=\kappa_\rho([x,y],z),
\end{equation}
for all $x,y,z\in\kg_\C$ and all $g\in G_\C$. In particular, $\kappa_\mathrm{ad}$ is called the \emph{Killing form}. It is non-degenerate on $\kg_\C$ and restricts to a \emph{negative} definite form on $\kg$ (while it is \emph{positive} definite on $i\kg$). In fact, bilinear forms on $\kg_\C$ satisfying the right-hand-side of \eqref{eq:killing} are unique modulo a positive multiplicative constant, and the standard convention is to use the normalisation $\tr:=\frac{2\kappa_\mathrm{ad}}{\kappa_\mathrm{ad}(\theta,\theta)}$, where $\theta$ is the longest root of $\kg_\C$ \cite[Chapter 2.4.1]{Frenkel04_book}. Once and for all, we fix an orthonormal basis $\kB\subset\kg_\C$ of $\kg_\C$ (over $\C$) with respect to $\tr$, i.e. $\tr(a,b)=\delta_{a,b}$ for all $a,b\in\kB$. Note that $\kB\subset i\kg$. We will frequently abuse notations by writing $\tr(ab)=\tr(a,b)$.

Given any $\rho$ as above, the \emph{Casimir operator} is 
\begin{equation}\label{eq:casimir}
\Omega_\rho:=\sum_{b\in\kB}\rho(b)\circ\rho(b)\in\mathrm{End}(V).
\end{equation}
Once we have fixed a normalisation of the Killing form (here, it is given by $\tr$ above), this expression is independent of the specific choice of orthonormal basis. The Casimir operator is central, i.e. $[\Omega_\rho,\rho(x)]=0$ for all $x\in\kg_\C$. This implies that $\Omega_\rho$ is a multiple of the identity whenever $\rho$ is an irreducible representation. In particular \cite[Chapter 3.4.8]{Frenkel04_book}, 
\begin{equation}\label{eq:coxeter}
\Omega_\mathrm{ad}=2\check h\mathrm{Id}_{\kg_\C},
\end{equation}
which can be taken as the definition of the dual Coxeter number $\check h\in\N$ as far as this article is concerned. In other words, one has $\sum_{b\in\kB}[b,[b,x]]=2\check hx$ for all $x\in\kg_\C$.

\begin{example}\label{example}
    Consider $G=\mathrm{SU}_N$ with complexification $G_\C=\mathrm{SL}_N(\C)$ and antiholomorphic involution $A\mapsto(A^*)^{-1}$ (the $*$-superscript denoting the Hermitian adjoint). The Lie algebra of $\mathrm{SU}_N$ is $\mathfrak{su}_N$, the (real) space of anti-Hermitian $N\times N$ matrices with vanishing trace. The complexification of $\mathfrak{su}_N$ is $\mathfrak{sl}_N(\C)$. The normalised Killing form is given by $\tr(A,B)=\mathrm{Tr}(AB)$, with $\mathrm{Tr}$ denoting the usual trace of matrices \cite[Chapter~2.4.1]{Frenkel04_book}. We have $\tr(A,B)=-\mathrm{Tr}(AB^*)$ on $\mathfrak{su}_N$, which is negative definite. The orthonormal basis for $\tr$ is a subset of $i\mathfrak{su}_N$, the space of Hermitian matrices with vanishing trace. The dual Coxeter number is $\check h=N$.
\end{example}

    \subsection{Loop groups}\label{subsec:loop_groups}
We will now recall some properties of loop groups, especially a well-known right-invariant K\"ahler structure. The main reference is \cite[Chapter 8.9]{Pressley-Segal}. See also \cite{Pressley82} and \cite{Segal_LG-talk} for concise accounts.

The \emph{analytic loop group} $L^\omega G$ is the space of maps $\gamma:\S^1\to G$ admitting a holomorphic continuation to a $G_\C$-valued function defined in an annular neighbourhood of $\S^1$. The Lie algebra of $L^\omega G$ is the space $L^\omega\kg$ of real-analytic $\kg$-valued functions on $\S^1$. We denote the subspace of mean-0 functions by
\begin{equation}\label{eq:loop-algebra}
L^\omega_0\kg:=\left\lbrace u\in L^\omega\kg\Big|\,\oint_{\S^1} u(z)\frac{\d z}{z}=0\right\rbrace.
\end{equation}
Canonical generators of $L^\omega_0\kg$ are given by $\{(ib\otimes(z^m+z^{-m}),\,b\otimes(z^m-z^{-m})),\,b\in\kB,m\geq1\}$; recall that $\kB\subset i\kg$. In other words, every $u\in L^\omega_0\kg$ has a Fourier expansion $u(z)=\sum_{b\in\kB,m\geq1}(u_{b,m}b\otimes z^m-\bar u_{b,m}b\otimes z^{-m})$ in the neighbourhood of $\S^1$, which we can alternatively view as the real part of the holomorphic function $\sum_{b\in\kB,m\geq1}u_{b,m}b\otimes z^m$ defined in the neighbourhood of $\bar\D$. 

The \emph{Kac--Moody cocycle} is the skew-symmetric bilinear form on $L^\omega_0\kg$ given by
\begin{equation}\label{eq:def-omega}
\omega_\km(u,v):=\frac{1}{2\pi}\oint_{\S^1}\tr(u\del v),\qquad\forall u,v\in L^\omega_0\kg.
\end{equation}
In the above coordinates, we have $\omega_\km(u,v)=2\Im(\sum_{b\in\kB,m\geq1}m\bar u_{b,m}v_{b,m})$. One sees from this expression that $\omega_\km$ is non-degenerate on $L^\omega_0\kg$. The Kac--Moody cocycle is a Lie algebra cocycle in the sense that
\begin{equation}\label{eq:cocycle}
\omega_\km(u,[v,w])+\omega_\km(v,[w,u])+\omega_\km(w,[u,v])=0,\qquad\forall u,v,w\in L^\omega_0\kg_\C.
\end{equation}

The \emph{Hilbert transform} is the linear complex structure on $L^\omega_0\kg$ defined by 
\begin{equation}\label{eq:def-J}
Ju(z)=\frac{1}{2\pi}\oint_{\S^1} u(\zeta)\frac{z+\zeta}{z-\zeta}\frac{\d\zeta}\zeta,\qquad\forall u\in L^\omega_0\kg,
\end{equation}
where the integral is in the principal value sense. We may identify the complexification of $L^\omega_0\kg$ with the space $L^\omega_0\kg_\C$ of $\kg_\C$-valued holomorphic functions in a neighbourhood of $\S^1$. The Kac--Moody cocycle and the Hilbert transform extend $\C$-linearly to the complexification. The $+i$-eigenspace (resp. $-i$-eigenspace) of $J$ is identified with the space $D^\omega_0\kg_\C$ (resp. $\hat D^\omega_\infty\kg_\C$) of $\kg_\C$-valued holomorphic functions in a neighbourhood of $\bar\D$ (resp. $\overline{\D^*}$) vanishing at 0 (resp. $\infty$). The Hilbert transform is an integrable complex structure in the sense that 
\begin{equation}\label{eq:J_integrable}
[u,v]+J([Ju,v]+[u,Jv])-[Ju,Jv]=0,\qquad\forall u,v\in L_0^\omega\kg_\C.
\end{equation}
The structures $J$ and $\omega_\km$ are compatible in the sense that $\omega_\km(Ju,Jv)=\omega_\km(u,v)$, and the $\R$-bilinear form
\[\langle u,v\rangle_\km:=\omega_\km(u,Jv),\qquad\forall u,v\in L^\omega_0\kg,\]
is positive definite, i.e. it induces an inner-product on $L^\omega_0\kg$. In the coordinates introduced above, we have $\langle u,v\rangle_\km=2\Re(\sum_{b\in\kB,m\geq1}m\bar u_{b,m}v_{b,m})$. This inner-product extends to a Hermitian form on $L^\omega_0\kg_\C$. By convention, we take the $\C$-linear part in the second variable. The complexification of $L^\omega\kg$ is the space $L^\omega\kg_\C$ of analytic $\kg_\C$-valued functions in the neighbourhood of $\S^1$, with generators $\{b_m:=b\otimes z^m|\,b\in\kB,m\geq1\}$.


We may view $\omega_\km$ as a right-invariant, non-degenerate 2-form on $L^\omega G/G$. The cocycle property \eqref{eq:cocycle} is then equivalent to the closedness of $\omega_\km$, turning $L^\omega G/G$ into a symplectic manifold. We may also view $J$ as a right-invariant complex structure on $L^\omega G/G$, and these two structures combine to make $L^\omega G/G$ a K\"ahler manifold \cite[Proposition (8.9.8)]{Pressley-Segal}.
 
\begin{remark}\label{rem:killing}
    The group $L^\omega G/G$ possesses a bi-invariant Riemannian metric induced by the inner-product is $(u,v)\mapsto\frac{i}{2\pi}\oint_{\S^1}\tr(u(z)v(z))\frac{\d z}{z}$ on $L^\omega_0\kg_\C$, which is different from the Kac--Moody metric. 
\end{remark}

    \subsection{Birkhoff factorisation}\label{subsec:factorisation}
This section recalls a factorisation formula for loops $\gamma:\S^1\to G$ which generalises the Birkhoff factorisation formula \cite[Chapter 8.1]{Pressley-Segal}. The main reference for this section is \cite{Pressley82}. 

Let $DG_\C$ be the group of holomorphic $G_\C$-valued functions on $\D$, and $D^\omega G_\C$ be the subgroup of those loops extending holomorphically to a neighbourhood of $\bar\D$. Let $D_0G_\C$ (resp. $D^\omega_0G_\C$) be the subgroup of those $g\in D_0G_\C$ (resp. $g\in D^\omega_0G_\C$) such that $g(0)=1_G$. Note that $D_0G_\C$ is topologically trivial: for each $g\in D_0G_\C$, the family $g_t(z):=g(e^{-t}z)$, $t\in[0,\infty]$, defines a homotopy from $g$ to the constant loop $\mathbf 1_G$. We also define $\hat D_\infty G_\C,\hat D_\infty^\omega G_\C$ in the same way but replacing $\D$ with $\D^*$ and 0 with $\infty$. Let $\cA$ be the space of holomorphic $\kg_\C$-valued $(1,0)$-forms in $\D$, and $\cA^\omega$ the subspace of those forms extending holomorphically to a neighbourhood of $\bar\D$.

By \cite[Theorem 1]{Pressley82} (see also \cite[Chapter 8]{Pressley-Segal} and \cite{Segal_LG-talk}), every loop $\gamma\in L^\omega G$ admits a factorisation 
\[\gamma=g\lambda\hat g^{-1}\]
with $g\in D^\omega G_\C$, $\hat g\in\hat D^\omega G_\C$, and $\lambda\in\mathrm{Hom}(\S^1;G)$. This factorisation is unique modulo the obvious global symmetry. The way this is proved is by studying the gradient flow of the kinetic energy \eqref{eq:kinetic}: the set $\mathrm{Hom}(\S^1;G)$ corresponds to the critical points of $\sE$, and the flow converges for every starting point $\gamma\in L^\omega G$ to a point $\lambda\in\mathrm{Hom}(\S^1;G)$. If we identify two based loops $\gamma_1,\gamma_2$ when their gradient flow converge to the same $\lambda$, we can then think of $\mathrm{Hom}(\S^1;G)$ as the quotient space, and we denote by $L^\omega_\lambda G$ the equivalence class of the homomorphism $\lambda$. In particular, $L^\omega_1G$ is the equivalence class of the trivial homomorphism. It is shown in Lemma~\ref{lem:subgroup} below that $L^\omega_1G$ is a subgroup of $L^\omega G$. 

Moreover, according to \cite[Proposition 2.3]{Segal_LG-talk} and the subsequent discussion, $L^\omega_1G$ can be identified with $D^\omega_\infty G_\C$ (denoted $L_0^-G_\C$ therein), or equivalently with $D^\omega_0G_\C$. In other words, for every $g\in D^\omega_0G_\C$, there exists $\hat g\in\hat D^\omega G_\C$ such that $g\hat g^{-1}$ is a real loop in $L^\omega G$, and this loop is unique if we impose $\hat g(\infty)=1_G$. 

We thus have an identification $L^\omega_1G\simeq D^\omega_0 G_\C$, and we consider now the map $D^\omega_0 G_\C\to\cA^\omega,\,g\mapsto g^{-1}\del g$. It is easy to see that this map is a bijection. Indeed, up to choosing a faithful representation, we may assume that $G_\C$ is a matrix group. Then, given $\alpha\in\cA^\omega$, we consider the complex ODE $g'(z)=\alpha(z)g(z)$ in $\D$, with initial condition $g(0)=1_{G_\C}$. Existence and uniqueness of a solution are immediate from the facts that the coefficient is globally Lipschitz and $\D$ is topologically trivial. 

Let us summarise these findings in a proposition.

\begin{proposition}\label{prop:factorisation}
    For each $\alpha\in\cA^\omega$, there exists a unique pair $(g,\hat g)\in D^\omega_0 G_\C\times\hat D^\omega_\infty G_\C$ such that $g\hat g^{-1}\in L^\omega_1G$ and $g^{-1}\del g=\alpha$.
\end{proposition}

It is useful to rephrase Proposition \ref{prop:factorisation} as a $\delbar$-problem. Any $\gamma\in L^\omega_1G$ is contractible, so we may extend it smoothly to a $G_\C$-valued function on $\hat\C$ such that it is holomorphic in the neighbourhood of $\S^1$, and satisfies the Schwarz reflection $\gamma(1/\bar z)=\overline{\gamma(z)}$. Then, we have 
\[\delbar\gamma\gamma^{-1}=\mu+\mu^*\qquad\text{in }\hat\C,\] 
where $\mu$ is a smooth $\kg_\C$-valued $(0,1)$-form compactly supported in $\D^*$, and $\mu^*(z)=-\bar z^{-2}\overline{\mu(1/\bar z)}$ (compactly supported in $\D$). On the other hand, if we write the factorisation $\gamma=g\hat g^{-1}$ and extend $g,\hat g$ to the whole sphere, the identity $\delbar\gamma\gamma^{-1}=\delbar gg^{-1}-\Ad_g(\hat g^{-1}\delbar\hat g)$ shows that $\delbar gg^{-1}=\mu$ and $\hat g^{-1}\delbar\hat g=-\Ad_{g^{-1}}(\mu^*)$ in $\hat\C$ (since $g,\hat g$ are holomorphic in $\D,\D^*$ respectively). 

Thus, the Birkhoff factorisation appears as the solution of the following general $\delbar$-problem, which is well-known in the complex analytic theory of moduli spaces of bundles \cite{Takhtajan-Zograf08}: given a smooth $\kg_\C$-valued $(0,1)$-form in $\hat\C$, solve for
\begin{equation}\label{eq:dbar}
\delbar gg^{-1}=\mu\text{ in }\hat\C\qquad\text{and}\qquad g(0)=1_{G_\C},
\end{equation}
where the unknown is a smooth function $g:\hat\C\to G_\C$. In the context of the Birkhoff factorisation, we have just seen that the two cases of interest are: either $\mu$ is reflection symmetric across the unit circle (so the solution is a real analytic loop), or $\mu$ is compactly supported in a disc (so the solution is holomorphic in the complementary disc). 
 
 The infinitesimal version of \eqref{eq:dbar} is to solve for
\begin{equation}\label{eq:cauchy_problem}
\delbar u=\mu\qquad\text{and}\qquad u(0)=0,
\end{equation}
where the unknown is now a smooth function $u:\hat{\C}\to\kg_\C$. This version of the problem has a well-known solution given by the Cauchy transform \cite[Chapter~V, Lemma~2]{Ahlfors66}
\begin{equation}\label{eq:cauchy_transform}
    u(z)=\frac{1}{\pi}\int_\C\mu(\zeta)\frac{z|\d\zeta|^2}{\zeta(z-\zeta)},\qquad\forall z\in\hat\C,
\end{equation}
where $|\d\zeta|^2=\frac{i}{2}\d\zeta\wedge\d\bar\zeta$ is the Lebesgue measure on $\C$. In the rest of the paper, we will mostly be concerned with the case where $\mu$ is compactly supported either in $\D$ or in $\D^*$.

\begin{lemma}\label{lem:subgroup}
The set $L^\omega_1G$ is a subgroup of $L^\omega G$. Moreover, it acts by left multiplication on each $L^\omega_\lambda G$. 
\end{lemma}
\begin{proof}
    Clearly, $L^\omega_1G$ contains $\ind_G$ and is stable under the inverse map. 
    
    Let $\gamma_1\in L^\omega_1G$ and $\gamma_2=g_2\lambda\hat g_2^{-1}\in L^\omega_\lambda G$. Since $\lambda$ is a homomorphism, there exists $x\in i\kg$ such that $e^{2i\pi x}=1_G$ and $\lambda(z)=z^x$ for all $z\in\S^1$. This expression has an analytic continuation to all $z\in\C\setminus\{0\}$. 
    
    Let $h\in\hat D^\omega_\infty G_\C$ be such that $\gamma_1g_2h^{-1}\in D_0^\omega G_\C$ (this can be exhibited as the solution of a $\delbar$-problem as above). Observe that $(\Ad_{\lambda^{-1}}h)(z)=z^{-x}h(z)z^x$ is bounded as $z\to\infty$, hence this function extends analytically around $z=\infty$ by Riemann's removabillity theorem. It follows that $\hat g_2\lambda h\lambda^{-1}\in\hat D^\omega_\infty G_\C$ as a product of two analytic functions in the neighbourhood of $\overline{\D^*}$. Then, we can write
    \[\gamma_1\gamma_2=\gamma_1g_2h^{-1}\lambda\lambda^{-1} h\lambda\hat g^{-1}_2=(\gamma_1g_2h^{-1})\lambda(\hat g_2\lambda^{-1}h\lambda)^{-1},\]
    showing that $\gamma_1\gamma_2\in L^\omega_\lambda G$. This implies both that $L^\omega_1G$ is a group and that the left multiplication by $L^\omega_1G$ preserves each $L^\omega_\lambda G$.
\end{proof}

\section{Complex representation of the Kac--Moody algebra}\label{sec:km_complex}

In this section, we define two commuting representations of the Kac--Moody algebra acting as differential operators on a space of polynomials on $\cA$. This representation is constructed by extending the action of $D^\omega_0G_\C$ on itself to an action of the complex loop group $L^\omega G_\C$. Identifying $D_0G_\C$ with $\cA$, the complex structure considered in this procedure is the canonical complex structure on $\cA$, which is different from the complex structure inherited from the Kac--Moody K\"ahler metric. This latter point of view will be the one adopted in Section \ref{sec:km_real}.  

Here and throughout the paper, we will frequently use the exponential notation in the Lie algebraic sense. For instance, given $u\in\hat D^\omega_0G_\C$, the exponential curve $(g_t=e^{tu})_{t\in\C}\in\hat D^\omega_0G_\C$ is the solution to the ODE (taking values in $\cA^\omega$ with the parametrisation $\alpha_t=g_t^{-1}\del g_t$)
\begin{equation}\label{eq:exp}
\del_t\alpha_t(z)=\del u(z)+[\alpha_t(z),u(z)]\qquad\text{with}\qquad\alpha_0=0,
\end{equation}
which is a one-dimensional ODE with globally Lipschitz coefficient for each $z$ in the neighbourhood of $\bar\D$. We will mostly be concerned with the expansion at order 1 as $t\to0$, in order to compute fundamental vector fields.

    \subsection{Coordinates and polynomials}\label{subsec:coord}
 Every $\alpha\in\cA$ admits a power series expansion around 0 of the form
\[\alpha(z)=\sum_{b\in\kB,m\geq1}\alpha_{b,m}b\otimes z^{m-1},\]
and we think of these coefficients as coordinates on $\cA$. We let $\cC:=\C[(\alpha_{b,m},\bar\alpha_{b,m})_{b\in\kB,m\geq1}]$ be the space of polynomials in the coefficients $(\alpha_{b,m})_{b\in\kB,m\geq1}$. We also write 
\begin{equation}\label{eq:def_polynomials}
\cC^{1,0}:=\C[(\alpha_{b,m})_{b\in\kB,m\geq1}]\qquad\text{and}\qquad\cC^{0,1}:=\C[(\bar\alpha_{b,m})_{b\in\kB,m\geq1}].
\end{equation}
Note that $\cC=\cC^{1,0}\otimes\cC^{0,1}$.

Now, we introduce a convenient way of coding elements of $\cC$ with integer partitions. Let $\cT$ be the set of sequences $\bk=(k_{b,m})_{b\in\kB,m\geq1}\in\N^{\kB\times{\N_{>0}}}$ with finitely many non-zero terms. The \emph{level} and \emph{length} of $\bk$ are defined respectively as
\[|\bk|:=\sum_{b\in\kB,m\geq1}mk_{b,m}\qquad\text{and}\qquad\ell(\bk):=\sum_{b\in\kB,m\geq1}k_{b,m}.\]
Let $\cT_N:=\{\bk\in\cT|\,|\bk|=N\}$. We will use the shorthand notation
\begin{equation}\label{eq:poly-alpha}
\boldsymbol{\alpha}^{\bk}:=\prod_{b\in\kB,m\geq1}\alpha_{b,m}^{k_{b,m}}\in\cC^{1,0}\qquad\text{and}\qquad\boldsymbol{\bar\alpha}^\bk:=\prod_{b\in\kB,m\geq1}\bar\alpha_{b,m}^{k_{b,m}}\in\cC^{0,1},
\end{equation}
for all $\bk\in\cT$. We write $\cC^{1,0}_N=\mathrm{span}\{\boldsymbol{\alpha}^{\bk}|\,\bk\in\cT_N\}$ and define similarly $\cC^{0,1}_N$. We have a grading 
\[\cC^{1,0}=\bigoplus_{N\in\N}\cC^{1,0}_N\qquad\text{and}\qquad\cC^{0,1}=\bigoplus_{N\in\N}\cC^{0,1}_N.\]

We stress that the coordinates $(\alpha_{b,m})_{b\in\kB,m\geq1}$ are holomorphic with respect to the canonical complex structure on $\cA$, but \emph{not} with respect to the induced Kac--Moody complex structure. This will play a hidden (albeit very concrete) role throughout the whole of Section \ref{sec:proof}.

    \subsection{Action of \texorpdfstring{$L^\omega G_\C$}{complex loops} on \texorpdfstring{$\cA$}{holomorphic forms}}\label{subsec:action_complex}
As a subgroup of $D_0G_\C$, the group $D_0^\omega G_\C$ acts on $D_0G_\C$ by right multiplication, inducing an action of $D^\omega_0G_\C$ on $\cA$ via $(\alpha,g)\mapsto g^{-1}\del g+g^{-1}\alpha g$. In this section, we will extend this to an action of the complex loop group $L^\omega_1G_\C$ (which contains $D^\omega_0G_\C$ as a subgroup) on $\cA$ and compute the corresponding fundamental vector fields. To ease our way into the construction of this extended action, consider a Birkhoff factorisation $\gamma=g\hat g^{-1}\in L^\omega_1G$ and let $h\in\hat D^\omega_\infty G_\C$. What is the Birkhoff factorisation of the function $\hat gh\in\hat D^\omega_\infty G_\C$? To answer this question, we look for $\chi\in L^\omega_1G$ such that $\chi^{-1}gh$ is holomorphic in $\D$, so that $\chi^{-1}\gamma=(\chi^{-1}gh)(\hat g h)^{-1}$ is the new Birkhoff factorisation. In particular, if $h$ is constant, then $\chi=h$ and the action of $h$ on $g$ is given by conjugation. 

Let $g\in D_0G_\C$ and $h\in L^\omega G_\C$ a contractible loop. Extend $h$ smoothly to a $G_\C$-valued function on $\hat\C$ (still denoted $h$), and define $\mu:=h^{-1}\delbar h\ind_{\D^*}$. Let $\chi_h:\hat\C\to G_\C$ be the solution to
\[\delbar\chi_h\chi_h^{-1}=\Ad_{gh}(\mu)+\Ad_{gh}(\mu)^*\text{ in }\hat\C\qquad\text{and}\qquad\chi_h(0)=1_G.\]
By symmetry across the unit circle, $\chi_h$ restricts to an element of $L^\omega_1G$. Moreover, an elementary computation gives $(\chi_h^{-1}gh)^{-1}=h^{-1}\delbar h-\Ad_{(gh)^{-1}}(\delbar\chi_h\chi_h^{-1})=0$ in $\D$, so that $\chi_h^{-1}gh\in D_0G_\C$. Now, let $h_1,h_2$ be two functions as above, and define $h:=h_1h_2$. 
The computation of each $\delbar$-derivative as above shows that $\chi_{h_1}\chi_{h_2}=\chi_h$, which implies that the map $(g,h)\mapsto\chi_h^{-1}gh$ is an action of $L^\omega_1G_\C$ on $D_0G_\C\simeq\cA$. Note that, if $h$ is constant, then $g\cdot h=h^{-1}gh$, i.e. the action restricts to the action of $G_\C$ by conjugation. Viewed as an action on $\cA$, we will denote this action by $(\alpha,h)\mapsto\alpha\cdot h$.

We are now going to differentiate this action at the identity, i.e. we exhibit the fundamental vector fields. By definition, for each $u\in L^\omega\kg_\C$, we look for an operator $Ku:\cA\to\cA^\omega$ associating to each $\alpha\in\cA$ the infinitesimal motion $K_\alpha u\in T_\alpha\cA\simeq\cA$ generated by $u$ at $\alpha$. Let then $u\in L^\omega\kg_\C$. We may extend $u$ to a smooth $\kg_\C$-valued functions on $\hat\C$. Let $h_t$ be any $C^1$-family in $L^\omega_1G_\C$ defined in a complex neighbourhood of $t=0$, such that
\[h_0=\ind_G,\qquad\del_th_t|_{t=0}=u,\qquad\del_{\bar t}h_t|_{t=0}=0.\]
Define $\mu_t:=h_t^{-1}\delbar h_t\ind_{\D^*}$. Let $\chi_t\in L^\omega_1G$ be such that $g_t:=\chi_t^{-1}gh_t\in D_0G_\C$. Let $\tilde u$ be the solution to 
\[\delbar\tilde u=\Ad_g(\delbar u)\text{ in }\hat\C\qquad\text{and}\qquad\tilde u(0)=0.\]
Then, we have $\del_t\chi_t|_{t=0}=\tilde u$ and $\del_{\bar t}\gamma_t|_{t=0}=\tilde u^*$, where we recall the notation $\tilde u^*(z)=\overline{\tilde u(1/\bar z)}$. Letting $\alpha_t:=g_t^{-1}\del g_t$, we have the first order expansion $\alpha_t=\alpha+tK_\alpha u+\bar tK^*_\alpha u+o(t)$ in $\cA$, with
\[K_\alpha u=\del u+[\alpha,u]-\Ad_{g^{-1}}(\del\tilde u)\qquad\text{and}\qquad K^*_\alpha u=-\Ad_{g^{-1}}(\del\tilde u^*).\] 
By construction, both $K_\alpha u$ and $K_\alpha^* u$ are holomorphic forms in a neighbourhood of $\bar\D$, i.e. $K_\alpha u,K_\alpha u^*\in\cA^\omega$. The next lemma gives an explicit expression of $K_\alpha u$.

\begin{lemma}\label{lem:K}
    For all $\alpha=g^{-1}\del g\in\cA$ and $u\in L^\omega\kg_\C$, we have
    \[K_\alpha u(z)=\frac{1}{2i\pi}\oint g^{-1}(z)g(\zeta)u(\zeta)g^{-1}(\zeta)g(z)\frac{\d\zeta}{(\zeta-z)^2},\qquad\forall z\in\D,\]
    where the contour is any loop in $\D$ contained in the domain of analyticity of $u$, and containing $z$ in its interior (e.g. $r\S^1$ for some sufficiently large $r\in(0,1)$).
\end{lemma}
\begin{proof}
We may extend $u$ smoothly to a $\kg_\C$-valued function on $\hat\C$. Let now $z$ outside the support of $\delbar u$. With the notation of the previous paragraph, we can use the Cauchy transform \eqref{eq:cauchy_transform} and Stokes' formula to get
\[\tilde u(z)
    =\frac{1}{\pi}\int_\D g(\zeta)\del_{\bar\zeta}u(\zeta)g^{-1}(\zeta)\frac{z|\d\zeta|^2}{\zeta(z-\zeta)}=\frac{1}{2i\pi}\oint g(\zeta)u(\zeta)g^{-1}(\zeta)\frac{z\d\zeta}{\zeta(z-\zeta)},\]
    where the contour contains the support of $\delbar u$ in its interior and $z$ in its exterior. We can deform this contour to have $z$ in the interior, at the cost of compensating for the residue at the pole $\zeta=z$. This gives
    \[\tilde u(z)=g(z)u(z)g^{-1}(z)+\frac{1}{2i\pi}\oint g(\zeta)u(\zeta)g^{-1}(\zeta)\frac{z\d\zeta}{\zeta(z-\zeta)},\]
    where now the contour is as in the statement of the lemma. Now, using the elementary identity $\Ad_{g^{-1}}\circ\del\circ\Ad_g(u)=\del u+[g^{-1}\del g,u]$, we have for all $z\in\D$ away from the support of $\delbar u$:
\begin{align*}
    K_\alpha u(z)
    &=\del u(z)+[g^{-1}\del g(z),u(z)]-\Ad_{g^{-1}(z)}(\del\tilde u(z))\\
    &=\frac{1}{2i\pi}\oint g^{-1}(z)g(\zeta)u(\zeta)g^{-1}(\zeta)g(z)\frac{\d\zeta}{(\zeta-z)^2}.
\end{align*}
This expression clearly defines a holomorphic function of $z\in\D$, lifting the initial condition that $z$ must be outside the support of $\delbar u$.
\end{proof}

 We end this section with a useful lemma, which is an elementary consequence of the previous one. Given $u,v\in L^\omega\kg_\C$, let us denote
 \begin{equation}\label{eq:def_P}
 P_u^v:=-\frac{1}{4\pi^2}\oint_{r\S^1}\oint_{R\S^1}\tr\left(g^{-1}(\zeta)u(\zeta)g(\zeta),g^{-1}(z)v^*(z)g(z)\right)\frac{\d\zeta\d z}{(\zeta-z)^2},
 \end{equation}
 where $0<r<R<1$ and both are large enough so that $u,v$ are analytic in that region. These encode the matrix coefficients of the operator $K_\alpha:L^\omega\kg_\C\to\cA^\omega$, in the sense that
 \[K_\alpha(b\otimes z^m)=\sum_{a\in\kB,n\geq1}P_{b,m}^{a,n}a\otimes z^{n-1}.\]
 Moreover, $P_{b,m}^{a,n}$ is obtained as from the Taylor coefficients of the power series expansions of $\Ad_g(a)$ and $\Ad_g(b)$, more precisely since $\zeta>z$, we can Taylor expand $(z-\zeta)^{-2}$ to get
 \[P_{b,m}^{a,n}=-\frac{1}{4\pi^2}\sum_{k=1}^\infty k\oint\oint\tr\left(g(\zeta)bg^{-1}(\zeta),g(z)\bar ag^{-1}(z)\right)\zeta^{m-k-1}z^{-n+k-1}\d\zeta\d z.\]
The contribution of the $k^{\text{th}}$ term in the series is given by expanding $\Ad_g(a)$ at order $(k-m)_+$ and $\Ad_g(b)$ at order $(n-k)_+$.\footnote{Here, we use the conventions that $x_+=x$ if $x\geq0$ and $x_+=-\infty$ if $x<0$, and a polynomial has level $-\infty$ if and only if it vanishes.} The coefficients of this expansion are polynomials in $\cC^{1,0}$ of level $(k-m)_+$ and $(n-k)_+$ respectively. Hence, whenever it is not zero, each contribution to $P_{b,m}^{a,n}$ is a polynomial of level $k-m+n-k=n-m$, so that $P_{b,m}^{a,n}\in\cC^{1,0}_{(n-m)_+}$. We gather this observation in a lemma. 

\begin{lemma}\label{lem:poly_K}
    Let $b\in\kB$, $m\in\Z$. We have $K_\alpha( b\otimes z^m)=\sum_{a\in\kB,n\geq1}P_{b,m}^{a,n}a\otimes z^{n-1}$, where $P_{b,m}^{a,n}\in\cC^{1,0}_{(n-m)_+}$ for all $a\in\kB$, $n\geq1$.
\end{lemma}

    \subsection{Differential operators and Kac--Moody representations}\label{subsec:km_complex}
 The next point is to exhibit two commuting representations of the Kac--Moody algebra acting as endomorphisms of $\cC$. We will start by a representation of the complexified loop algebra obtained as directional derivatives in the directions of the action defined in the previous section.

\begin{definition}[Differentiability along the complex action]\label{def:diff-complex}
 Let $\alpha\in\cA$ and $u\in L^\omega\kg_\C$. A function $F:\cA\to\C$ is \emph{differentiable at $\alpha$ in direction $u$} if there exist $\cJ_uF(\alpha),\bar\cJ_uF(\alpha)\in\C$ such that the following holds. For any $C^1$-family $(h_t)\in L^\omega G_\C$ defined in a complex neighbourhood of $t=0$ with
 \begin{equation}\label{eq:motion_ht}
 h_0=\ind_G,\qquad\del_th_t|_{t=0}=u,\qquad\del_{\bar t}h_t|_{t=0}=0,
 \end{equation}
 we have the first order expansion
 \[F(\alpha_t)=F(\alpha)+t\cJ_uF(\alpha)+\bar t\bar\cJ_uF(\alpha)+o(t),\]
 where $\alpha_t=g_t^{-1}\del g_t$ and $g_t=\chi_t^{-1}gh_t$ is defined using the action from Section \ref{subsec:action_complex}. We say that $F$ is differentiable in direction $u$ if it is differentiable in direction $u$ at every $\alpha\in\cA$.
\end{definition}
 
 From Lemma \ref{lem:poly_K}, $\cJ_u$ and $\bar\cJ_u$ are first order differential operators with polynomial coefficients in the coordinates $(\alpha_{b,m},\bar\alpha_{b,m})$. When acting on polynomials, only a finite number of differential operators contribute (since such functions depend on finitely many coordinates), so $\cJ_u,\bar\cJ_u$ are well-defined as endomorphisms of $\cC$.

 \begin{lemma}\label{lem:commute-cJ}
 As endomorphisms of $\cC$, we have for all $u,v\in\kg_\C(z)$:
     \[[\cJ_u,\cJ_v]=\cJ_{[u,v]},\qquad[\bar\cJ_u,\bar\cJ_v]=\bar\cJ_{[u,v]},\qquad[\cJ_u,\bar\cJ_v]=0.\]
 \end{lemma}
 \begin{proof}
     Let us denote $\cJ_u^\R:=\cJ_u+\bar\cJ_u$. These operators are the Lie derivatives along the right action of the complex loop group $L^\omega G_\C$ on $\cA$. Hence, they represent the complex loop algebra, i.e. $[\cJ_u^\R,\cJ_v^\R]=\cJ_{[u,v]}^\R$ for all $u,v\in\kg_\C(z)$. The result for the operators $\cJ_u,\bar\cJ_u$ follows by direct computation, e.g. using $\cJ_u=\frac{1}{2}(\cJ_u^\R-i\cJ_{iu}^\R)$, we find
     \begin{align*}
         [\cJ_u,\cJ_v]=\frac{1}{4}[\cJ_u^\R-i\cJ^\R_{iu},\cJ^\R_v-i\cJ^\R_{iv}]&=\frac{1}{4}\cJ_{[u,v]}^\R-\frac{1}{4}\cJ^\R_{[iu,iv]}-\frac{i}{4}\cJ_{[iu,v]}^\R-\frac{i}{4}\cJ^\R_{[u,iv]}\\
         &=\frac{1}{2}\cJ^\R_{[u,v]}-\frac{i}{2}\cJ_{i[u,v]}^\R=\cJ_{[u,v]}.
     \end{align*}
     We used the $\C$-linearity of the commutator in $\mathrm{End}(\cC)$, the $\C$-linearity of the Lie bracket in $\kg_\C(z)$, and the $\R$-linearity of the map $\kg_\C(z)\to\mathrm{End}(\cC),\,u\mapsto\cJ_u^\R$. The other commutators go the same way and are omitted.
 \end{proof}

Recall that we denote $\alpha=g^{-1}\del g\in\cA$, and we view $\alpha$ as a linear form on $L^\omega\kg_\C$ using the residue paring $\alpha(u)=\frac{1}{2i\pi}\oint\tr(u\alpha)$. We introduce the following modification of the previous operators: for all $u\in L^\omega\kg_\C$, set
\begin{equation}\label{eq:def-bJ}
\bJ_u:=\cJ_u+\kappa\alpha(u)\mathrm{Id}_\cC\qquad\text{and}\qquad\bar\bJ_u:=\bar\cJ_u+\kappa\overline{\alpha(u)}\mathrm{Id}_\cC.
\end{equation}
Given $b\in\kB$, $m\in\Z$, we will use the shorthand $\bJ_{b,m}=\bJ_{b\otimes z^m}$.

 \begin{proposition}\label{prop:km_complex}
     For all $u,v\in L^\omega\kg_\C$, we have
     \[\cJ_u\alpha(v)-\cJ_v\alpha(u)-\alpha([u,v])=i\omega_\km(u,v)\qquad\text{and}\qquad\bar\cJ_u\alpha(v)-\cJ_v\overline{\alpha(u)}=0.\]
     Hence, as endomorphisms of $\cC$, we have for all $u,v\in\kg_\C(z)$
     \[[\bJ_u,\bJ_v]=\bJ_{[u,v]}+i\kappa\omega_\km(u,v)\mathrm{Id}_\cC,\qquad[\bar\bJ_u,\bar\bJ_v]=\bar\bJ_{[u,v]}-i\kappa\overline{\omega_\km(u,v)}\mathrm{Id}_\cC,\qquad[\bJ_u,\bar\bJ_v]=0.\]
 \end{proposition}
Specialising to the generators $(\bJ_{b,m})_{b\in\kB,m\in\Z}$ and using that $\omega_\km(a\otimes z^n,b\otimes z^m)=-in\delta_{a,b}\delta_{n,-m}$ by definition, we get
\[[\bJ_{a,n},\bJ_{b,m}]=\bJ_{[a,b],n+m}+\kappa n\delta_{a,b}\delta_{n,-m}\mathrm{Id}_\cC.\]
These relations coincide with \cite[(15.45)]{CFTbook}, so we have two commuting representations of the Kac--Moody algebra at level $\kappa$.
 
 \begin{proof}
     Using Lemma \ref{lem:K}, we have
     \begin{align*}
         \cJ_u\alpha(v)
         &=\frac{1}{2i\pi}\oint\tr(v,K_\alpha u)=-\frac{1}{4\pi^2}\oint_{r\S^1}\oint_{R\S^1}\tr(g(z)v(z)g^{-1}(z),g(\zeta)u(\zeta)g^{-1}(\zeta))\frac{\d\zeta\d z}{(\zeta-z)^2},
     \end{align*}
     for any $0<r<R<1$. By symmetry of the integrand, the difference $\cJ_u\alpha(v)-\cJ_v\alpha(v)$ is just obtained by exchanging the two contours. Elementary contour deformation tells us that this is just given by the residue at the singularity $\zeta=z$, i.e.
     \begin{align*}
         \cJ_u\alpha(v)-\cJ_v\alpha(u)
         &=\frac{1}{2i\pi}\oint_{r\S^1}\tr\left(g(z)v(z)g^{-1}(z),\underset{\zeta=z}{\mathrm{Res}}\,g(\zeta)u(\zeta)g^{-1}(\zeta)\frac{\d\zeta}{(\zeta-z)^2}\right)\d z\\
         &=\frac{1}{2i\pi}\oint_{r\S^1}\tr\left(g(z)v(z)g^{-1}(z),\del_z\left(g(z)u(z)g^{-1}(z)\right)\right)\d z\\
         &=\frac{1}{2i\pi}\tr\left(v,\del u+[g^{-1}\del g,u]\right)=i\omega_\km(u,v)+\alpha([u,v]).
     \end{align*}
     In the last line, we used the invariance of the trace and the elementary identity $\Ad_{g^{-1}}\circ\del\circ\Ad_g(u)=\del u+[g^{-1}\del g,u]$. This concludes the proof of the first identity.

     For the second identity, observe first that $\bar\cJ_u\alpha(v)$ vanishes unless both $u,v\in\hat D^\omega_\kg\C$. The statement can then be rephrased as the fact that the sesquilinear form $(u,v)\mapsto\bar\cJ_u\alpha(v)$ on $\hat D^\omega_\infty\kg_\C$ is conjugate symmetric. By definition, we have $\bar\cJ_u\alpha(v)=\frac{1}{2i\pi}\oint_{\S^1}\tr(vK_\alpha^*u)$. Using the invariance of the Killing form \eqref{eq:killing} and an integration by parts, one see that this is proportional to $\oint_{\S^1}\tr(\Ad_g(u)^*\del\Ad_g(v))$, which is manifestly conjugate symmetric. 
     
     

    From here, the commutation relations follow from the Leibniz rule and elementary algebra.
 \end{proof}

    \subsection{Highest-weight representation}\label{subsec:hw}
Finally, we define a highest-weight representation of the pair of representations $(\bJ_u,\bar\bJ_u)_{u\in L^\omega\kg_\C}$, which will provide a natural basis of the space of polynomials $\C[(\alpha_{b,m},\bar\alpha_{b,m})_{b\in\kB,m\geq1}]$.

 Choose an arbitrary total order on $\kB$, and consider the lexicographic order on $\kB\times\N_{>0}$, whereby $(b,m)\leq(b',m')$ if and only if $m<m'$ or $m=m'$ and $b\leq b'$. Recall the definition of the set $\cT$ of integer partitions introduced in Section~\ref{subsec:coord}. Given such a partition $\bk=(k_{b,m})_{b\in\kB,m\geq1}$, we define the following endomorphisms of of $\cC$ 
\begin{equation}\label{eq:def-big-circ}
\bJ^{\circ\bk}:=\bigcirc_{b\in\kB,m\geq1}\bJ_{b,-m}^{k_{b,m}}\qquad\text{and}\qquad\bar\bJ^{\circ\bk}:=\bigcirc_{b\in\kB,m\geq1}\bar\bJ_{b,-m}^{k_{b,m}},
\end{equation}
where the composition of operators is ordered in decreasing lexicographical order from left to right. Now, define 
\begin{equation}\label{eq:def_W}
\cW:=\mathrm{span}\left\lbrace\Psi_{\bk,\tilde\bk}:=\bJ^{\circ\bk}\bar\bJ^{\circ\tilde\bk}\ind|\,\bk,\tilde\bk\in\cT\right\rbrace\subset\cC,
\end{equation}
where the linear span is in the algebraic sense, and $\ind$ denotes the constant function $1$ on $\cA$. From Lemma \ref{lem:poly_K}, $\Psi_{\bk,\tilde\bk}$ is a polynomial of level $(|\bk|,|\tilde\bk|)$. We also introduce the following subspace of~$\cC^{1,0}$:
\[\cV:=\mathrm{span}\left\lbrace\Psi_{\bk,\emptyset}|\,\bk\in\cT\right\rbrace\subset\cC^{1,0}.\]
 Standard representation theory tells us that $\cV$ is the quotient of the Verma module by a submodule. According to the Kac--Kazhdan classification \cite{Kac-Kazhdan_HW} (see also \cite[Theorem 1.1]{Feigin-Fuchs}), the Verma module is irreducible for this weight, hence $\cV$ is a Verma module. In particular, all the states $(\Psi_{\bk,\emptyset})_{\bk\in\cT}$ are linearly independent. Moreover, since $\Psi_{\bk,\emptyset}\in\oplus_{n=0}^{|\bk|}\cC^{1,0}_n$, dimension counting shows that the states $(\Psi_{\bk,\emptyset})_{|\bk|\leq N}$ actually span the whole of $\oplus_{n=0}^N\cC^{1,0}_n$, so that $\cV$ is linearly isomorphic to $\cC^{1,0}$ and $\cC$ admits the basis $\{\Psi_{\bk,\emptyset}|\,\bk\in\cT\}$. Finally, $\cW$ is the the tensor product of two independent copies of $\cV$, so that the states $\{\Psi_{\bk,\tilde\bk}|\,\bk,\tilde\bk\in\cT\}$ are linearly independent and furnish a basis of $\cC=\C[(\alpha_{b,m},\bar\alpha_{b,m})_{b\in\kB,m\geq1}]$.

 \section{Real representation of the Kac--Moody algebra}\label{sec:km_real}

The goal of this section is twofold. First, we construct a representation of the Kac--Moody algebra corresponding to the action of $L^\omega_1G$ on $\cA$ by left multiplication (viewing $\cA^\omega$ as a loop group orbit), see Section \ref{subsec:real_action}. Second, the Kac--Moody representation exhibits a $\del$-antiderivative of the Kac--Moody K\"ahler form, which we further integrate to a K\"ahler potential $\sS$ on $\cA^\omega$ (Section \ref{subsec:potential}). This potential will play the role of the action in the path integral. Importantly, the representation of Proposition \ref{prop:km_real} will turn out to be unitary on $L^2(\nu_\kappa)$, see Theorem \ref{thm:unitary}.

    \subsection{Left action of \texorpdfstring{$L^\omega_1G$}{real loops} on \texorpdfstring{$\cA$}{holomorphic forms}}\label{subsec:real_action}

Let $\alpha=g^{-1}\del g\in\cA$. Let $\chi\in L^\omega_1G$ and extend it smoothly to a $G_\C$-valued function on $\hat\C$ (still denoted $\chi$). Doing the construction of Section \ref{subsec:action_complex}, there exists a smooth function $h:\hat\C\to G_\C$, holomorphic in a neighbourhood of $\overline{\D^*}$, such that $\chi gh^{-1}\in D_0G_\C$. Moreover, the map $(\chi,g)\mapsto\chi gh^{-1}$ defines a left action of $L^\omega_1G$ on $\cA$. If we restrict this action to $\cA^\omega$ and view $g$ as an element of $L^\omega_1G$ using the Birkhoff factorisation of Section \ref{subsec:factorisation}, the action just described corresponds to the action of $L_1^\omega G$ on itself by left multiplication. 

We will now express the fundamental vector fields associated with this action. Let $u\in D_0^\omega\kg_\C$ and extend it smoothly to a smooth $\kg_\C$-valued function on $\hat\C$ (still denoted $u$). Let $(\chi_t)$ be a $C^1$-family in $L^\omega_1G$ defined in a complex neighbourhood of $t=0$ such that
\[\chi_0=\ind_G,\qquad\del_t\chi_t|_{t=0}=u,\qquad\del_{\bar t}\chi_t|_{t=0}=u^*,\]
where we recall $u^*(z)=\overline{u(1/\bar z)}$. By definition, the function $h_t$ such that $g_t:=\chi_tgh_t^{-1}\in D_0^\omega G_\C$ satisfies
\[h_0=\ind_G,\qquad\del_th_t|_{t=0}=0,\qquad\del_{\bar t}h_t|_{t=0}=\tilde u,\]
where $\tilde u$ is the unique solution to $\delbar\tilde u=\Ad_{g^{-1}}(\delbar u^*)$ in $\hat\C$, and $\tilde u(0)=0$. Hence, writing $\ralpha_t=-\del g_tg_t^{-1}$, we have the following expansion (recall the definition of $Ku$ from Lemma \ref{lem:K})
\begin{equation}\label{eq:vf_L}
\ralpha_t=\ralpha+tK_\ralpha u+\bar tK_\ralpha u^*+o(t)\qquad\text{in }\cA.
\end{equation}
This computes the fundamental vector field of the left action. We will now proceed as in Section~\ref{subsec:km_complex} and define differential operators as Lie derivatives along the directions of this action.

\begin{definition}\label{def:test_function}
A function $F:\cA\to\C$ is \emph{left-differentiable in direction $u\in D^\omega_0\kg_\C$} at $\alpha\in\cA$ if there exist $\cL_uF(\alpha),\cL_{u^*}F(\alpha)\in\C$ such that for all $C^1$-families $(\chi_t)\in L^\omega_1G$ defined in a complex neighbourhood of $t=0$ with
\[\chi_0=\ind_G,\qquad\del_t|_{t=0}\chi_t=u,\qquad\del_{\bar t}|_{t=0}\chi_t=u^*,\]
we have a first order expansion
\[F(\chi_t^{-1}\cdot\alpha)=F(\alpha)+t\cL_uF(\alpha)+\bar t\cL_{u^*}F(\alpha)+o(t).\]

We say that $F$ is left-differentiable in direction $u$ if it is left-differentiable in direction $u$ at every $\alpha\in\cA$. We say that $F$ is of class $C^1$ if it is left-differentiable in all directions $u\in D^\omega_0\kg_\C$, and $\cL_uF\in C^0(\cA)$ for all $u\in L^\omega\kg_\C$. We say that $F$ is smooth (and write $F\in C^\infty(\cA)$) if it is of class $C^1$ and $\cL_uF$ is smooth for all $u\in L^\omega\kg_\C$. We denote by $C^\infty_b(\cA)\subset C^\infty(\cA)$ the subspace of bounded functions, and by $C^\infty_c(\cA)\subset C^\infty_b(\cA)$ the subspace of compactly supported functions, a.k.a. the \emph{test functions}.
\end{definition}

From \eqref{eq:vf_L} and Lemma \ref{lem:poly_K}, it is clear that coordinate functions are smooth in the sense above, which readily implies that $C^\infty(\cA)$ is dense in $C^0(\cA)$. The definition introduces both the space $C^\infty(\cA)$ and a family of operators $(\cL_u,\cL_{u^*})_{u\in D^\omega_0\kg_\C}$ acting as endomorphisms of $C^\infty(\cA)$. Since the map $D^\omega_0\kg_\C\to\hat D^\omega_\infty\kg_\C,u\mapsto u^*$ is a linear isomorphism, we have a family $(\cL_u)_{u\in L^\omega_0\kg_\C}$. We have also noted earlier that constant loops act by conjugation; this also defines $\cL_u$ for constant $u\in\kg$, which we extend complex linearly for $u\in\kg_\C$. By construction, the assignment $u\mapsto\cL_u$ is complex linear, i.e. $\cL_{iu}=i\cL_u$ for all $u\in L^\omega\kg_\C$. 


\begin{lemma}\label{lem:commute-cL}
    As endomorphisms of $C^\infty(\cA)$, the operators $(\cL_u)_{u\in L^\omega\kg_\C}$ form a representation of $L^\omega\kg_\C$: for all $u,v\in\kg_\C(z)$, we have
    \[[\cL_u,\cL_v]=\cL_{[u,v]}.\]
\end{lemma}
\begin{proof}
    For all $u=u_++u_-$ with $u_+\in z\kg_\C[z]$ and $u_-=u_+^*$, set $\cL_u^\R:=\cL_{u_+}+\cL_{u_-}$. The operators $\cL_u^\R$ are the \emph{real} Lie derivatives along the action defined in the beginning of this section, hence they represent the \emph{real} loop algebra, i.e. $[\cL_u^\R,\cL_v^\R]=\cL_{[u,v]}^\R$ for all $u,v\in L^\omega\kg$. 

    To extend these relations to the complex Lie derivatives, we use the explicit expression $\cL_{u_\pm}=\frac{1}{2}(\cL_u^\R\mp i\cL_{Ju}^\R)$, where $J$ is the complex structure introduced in Section~\ref{subsec:loop_groups}. Using \eqref{eq:J_integrable}, we have 
    \begin{equation}\label{eq:C-ccommute-L}
    \begin{aligned}
        [\cL_{u_+},\cL_{v_+}]
        &=\frac{1}{4}\cL_{[u,v]}^\R-\frac{1}{4}\cL^\R_{[Ju,Jv]}-\frac{i}{4}\cL^\R_{[u,Jv]}-\frac{i}{4}\cL^\R_{[Ju,v]}\\
        &=\frac{1}{4}\cL^\R_{[u,v]}-\frac{i}{4}\cL^\R_{J[u,v]}-\frac{1}{4}\cL_{[Ju,Jv]}^\R+\frac{i}{4}\cL^\R_{J[Ju,Jv]}
        =\frac{1}{2}\cL_{[u,v]_+}-\frac{1}{2}\cL_{[Ju,Jv]_+}.
    \end{aligned}
    \end{equation}
Moreover, using again \eqref{eq:J_integrable}, we have
\begin{align*}
[u_+,v_+]&=\frac{1}{4}([u,v]-[Ju,Jv]-i[Ju,v]-i[u,Jv])\\
&=\frac{1}{4}([u,v]-iJ[u,v]-[Ju,Jv]+iJ[Ju,Jv])=\frac{1}{2}[u,v]_+-\frac{1}{2}[Ju,Jv]_+.
\end{align*}
Plugging this into \eqref{eq:C-ccommute-L}, we get indeed $[\cL_{u_+},\cL_{v_+}]=\cL_{[u_+,v_+]}$. The commutator $[\cL_{u_-},\cL_{v_-}]$ is computed in a similar fashion. As for the commutator $[\cL_{u_-},\cL_{v_+}]$, we have on the one hand
\[[\cL_{u_++u_-}^\R,\cL_{v_++v_-}^\R]=\cL^\R_{[u_++u_-,v_++v_-]}=\cL_{[u_+,v_+]}+\cL_{[u_+,v_-]}+\cL_{[u_-,v_+]}+\cL_{[u_-,v_-]},\]
and on the other hand 
\[[\cL_{u_++u_-}^\R,\cL_{v_++v_-}^\R]=[\cL_{u_+},\cL_{v_+}]+[\cL_{u_+},\cL_{v_-}]+[\cL_{u_-},\cL_{v_+}]+[\cL_{u_-},\cL_{v_-}],\]
and combining with the already known commutators, we find
\[\cL_{[u_+,v_-]}+\cL_{[u_-,v_+]}=[\cL_{u_+},\cL_{v_-}]+[\cL_{u_-},\cL_{v_+}].\]
Doing the same computation with $i(u_+-u_-)$ in place of $u_++u_-$, we also find
\[\cL_{[u_+,v_-]}-\cL_{[u_-,v_+]}=[\cL_{u_+},\cL_{v_-}]-[\cL_{u_-},\cL_{v_+}].\]
Combining these last two displayed equations gives $[\cL_{u_+},\cL_{v_-}]=\cL_{[u_+,v_-]}$ and $[\cL_{u_-},\cL_{v_+}]=\cL_{[u_-,v_+]}$ as required.
    \end{proof}

We are now going to introduce a modification of these operators so that they represent the Kac--Moody algebra with dual level $\check\kappa=-2\check h-\kappa$. To this end, given $\alpha\in\cA$, we define $\ralpha:=-\del gg^{-1}\in\cA^\omega$. We view it as a linear form on $L^\omega\kg_\C$ via the residue pairing:
\[\ralpha(u)=\frac{1}{2i\pi}\oint\tr(u\ralpha),\qquad\forall u\in L^\omega\kg_\C.\]
 Now, we introduce the operators
\begin{equation}\label{eq:def-bL}
\bL_u:=\cL_u+\check\kappa\ralpha(u)\mathrm{Id}\qquad\forall u\in L^\omega\kg_\C.
\end{equation}
\begin{proposition}\label{prop:km_real}
    For all $u,v\in L^\omega\kg_\C$ and $\alpha\in\cA$, we have
    \begin{equation}\label{eq:diff_ralpha}
    \cL_u\ralpha(v)-\cL_v\ralpha(u)-\ralpha([u,v])=i\omega_\km(u,v).
    \end{equation}
    Hence, as endomorphisms of $C^\infty(\cA)$, the operators $(\bL_u)_{u\in \kg_\C(z)}$ form a representation of the Kac--Moody algebra at level $\check\kappa$:
    \[[\bL_u,\bL_v]=\bL_{[u,v]}+i\check\kappa\omega_\km(u,v)\mathrm{Id}.
    \]
\end{proposition}

\begin{proof}
    Let $u\in L^\omega\kg_\C$ viewed as a smooth $\kg_\C$-valued function on $\hat{\C}$, and let us look at the small motion generated by the multiplication to the left. It is of the form $g_t=e^{-tu-\bar tu^*}ge^{\bar t\tilde u}$ where $\tilde u$ is the solution to $\delbar\tilde u=\Ad_{g^{-1}}(\delbar u^*)$. We then have $g_t^{-1}=e^{-\bar t\tilde u}g^{-1}e^{tu+\bar tu^*}$ and the first order expansion $\ralpha_t=\ralpha+tK_{\ralpha}u+\bar tK_{\ralpha}u^*+o(t)$ in $\cA$. It follows by definition that $\cL_u\ralpha(v)=\frac{1}{2i\pi}\oint\tr(vK_\ralpha u)$, so that applying the same argument as in the proof of Proposition \ref{prop:km_complex} gives $\cL_u\ralpha(v)-\cL_v\ralpha(u)=\ralpha([u,v])+i\omega_\km(u,v)$, and the same holds with $u^*,v^*$ in place of $u,v$.
\end{proof}

        \subsection{Right action of \texorpdfstring{$L^\omega_1G$}{real loops} on \texorpdfstring{$\cA$}{holomorphic forms}}\label{subsec:right}
In the previous section, we have defined a left action of $L^\omega_1G$ on $\cA$ whose restriction to $\cA^\omega$ corresponds to the left multiplication of $L^\omega_1G$ on itself under the Birkhoff factorisation. We could of course consider the right multiplication of $L^\omega_1G$ on itself, but this action does not extend to the whole of $\cA$ since we may not have a unique factorisation. However, we can still define a right action of $L^\omega_1G$ on $\cA^\omega$ and introduce the corresponding Lie derivatives.

Let $\chi\in L^\omega_1G$. To express the action of $\chi$ on $\alpha=g^{-1}\del g\in\cA^\omega$, we first let $\hat g$ be the $G_\C$-valued holomorphic function in a neighbourhood of $\overline{\D^*}$ such that $g\hat g^{-1}\in L^\omega_1G$. Then, applying the construction of Section \ref{subsec:action_complex} to $\hat g$, there is $h_\chi:\hat\C\to G_\C$ holomorphic in a neighbourhood of $\bar\D$ such that $h_\chi^{-1}\hat g^{-1}\chi$ is holomorphic in a neighbourhood of $\overline{\D^*}$. Moreover, the restriction of $h_\chi$ to $\D$ is uniquely determined. Then, $gh_\chi\in D_0^\omega G_\C$ and $g\hat g^{-1}\chi=(gh_\chi)(h_\chi^{-1}\hat g^{-1}\chi)$ is our new factorisation. The map $(g,\chi)\mapsto gh_\chi$ defines a right action of $L^\omega_1G$ on $D_0^\omega G_\C\simeq\cA^\omega$.

Now, we differentiate the action at $\chi=\ind_G$, i.e. we consider $\chi_t=e^{tu+\bar tu^*}$ for some $u\in D^\omega_0\kg_\C$ and small complex $t$. The function $h_t$ such that $h_t^{-1}\hat g^{-1}\chi_t$ is holomorphic in $\D^*$ satisfies
\[h_0=\ind_G,\qquad\del_th_t|_{t=0}=\tilde u,\qquad\del_{\bar t}h_t=0,\]
where $\tilde u$ is the solution to $\delbar\tilde u=\Ad_{\hat g^{-1}}(\delbar u)$ in $\hat\C$ and $\tilde u(0)=0$. Hence, writing $\alpha_t=g_t^{-1}\del g_t$, we have the first order expansion
\[\alpha_t=\alpha+tK_\alpha\tilde u+o(t)\text{ in }\cA.\]
For all $F\in C^\infty(\cA^\omega)$, we then have a first order expansion
\[F(\alpha\cdot\chi_t)=F(\alpha)+t\cR_uF(\alpha)+\bar t\cR_{u^*}F(\alpha)+o(t),\]
defining the operators $\cR_u\in\mathrm{End}(C^\infty(\cA^\omega))$ for all $u\in L^\omega\kg_\C$ (by convention, $\cR_u=0$ if $u$ is constant). Since the left and right actions commute, so do the Lie derivatives, i.e. $[\cL_u,\cR_v]=0$ for all $u,v\in L^\omega\kg_\C$. We stress again that the operators $(\cR_u)_{u\in L^\omega\kg_\C}$ cannot be extended to differential operators on the whole $\cA$, but only on the space $\cA^\omega$ where we can use the Birkhoff factorisation. See Section \ref{subsec:shapo} for a discussion on possible extensions. 

As for the left action, we introduce a modification of these operators:
\[\bR_u:=\cR_u+\check\kappa\hat\ralpha(u)\mathrm{Id}\qquad\forall u\in L^\omega\kg_\C.\]
As in the previous section, these operators form a representation of the Kac--Moody algebra at level $\check\kappa$. Moreover, the same computations show that $\cL_u\hat\ralpha(v)-\cR_v\ralpha(u)=0$ for all $u,v\in L^\omega\kg_\C$, implying that the two representations commute. We collect these observations in a proposition.

\begin{proposition}
    As endomorphisms of $C^\infty(\cA^\omega)$, the family $(\bL_u,\bR_u)_{u\in\kg_\C(z)}$ forms two commuting representations of the Kac--Moody algebra at level $\check\kappa$: for all $u,v\in\kg_\C(z)$, we have
    \[[\bL_u,\bL_v]=\bL_{[u,v]}+i\check\kappa\omega_\km(u,v),\qquad[\bR_u,\bR_v]=\bR_{[u,v]}+i\check\kappa\omega_\km(u,v),\qquad[\bL_u,\bR_v]=0.\]
\end{proposition}

    \subsection{Definition and properties of \texorpdfstring{$\sS$}{the K\"ahler potential}}\label{subsec:potential}

The operators $(\cL_u)_{u\in L^\omega\kg_\C}$ introduced in the previous section describe the differentiable structure of $\cA$ as a complex manifold endowed with the right-invariant Kac--Moody metric. We stress that this complex structure is not the same as the canonical complex structure on $\cA$, namely the coordinate functions $(\alpha_{b,m})_{b\in\kB,m\geq1}$ are not holomorphic with respect to the Kac--Moody metric. The operators $(\cL_u)_{u\in D^\omega_0\kg_\C}$ represent the $(1,0)$-part of the de Rham differential, while the operators $(\cL_u)_{u\in\hat D^\omega_\infty\kg_\C}$ represent the $(0,1)$-part. More precisely, given a test function $F$, its $(1,0)$-differential is the function $\del F:\cA\times D^\omega_0\kg_\C$ defined by $\del F_\alpha(u)=\cL_uF(\alpha)$. Similarly, $\delbar F_\alpha(u)=\cL_uF(\alpha)$ for all $\alpha\in\cA$ and $u\in\hat D^\omega_\infty\kg_\C$. 

Proposition \ref{prop:km_real} implies that, viewed as a $(0,1)$-form on $\cA^\omega$, $\ralpha=-\del gg^{-1}$ is $\delbar$-closed (since $\omega_\km$ vanishes on $\hat D^\omega_\infty\kg_\C\times\hat D^\omega_\infty\kg_\C$). In finite dimensional complex geometry, the $\delbar$-Poincar\'e lemma would immediately imply that $\ralpha$ is locally $\delbar$-exact (and in fact globally $\delbar$-exact since $\cA$ is topologically trivial). The next lemma/definition makes this precise in our infinite dimensional setup, by exhibiting a function $\sS$ such that $\delbar\sS=-\ralpha$. By Proposition \ref{prop:km_real}, we then have 
\[i\del\delbar\sS=\omega_\km,\]
i.e. $\sS$ is a K\"ahler potential for the Kac--Moody metric.

\begin{proposition}\label{prop:potential}
    There exists a unique function $\sS:\cA^\omega\to\R$ satisfying
    \begin{itemize}
    \item $\sS(\mathbf 1_G)=0$. 
    \item For all $u\in\hat D^\omega_\infty\kg_\C$, $\sS$ is left-differentiable in direction $u$, and $\cL_u\sS(g)=\frac{1}{2i\pi}\oint\tr(u\del gg^{-1})$.
    \end{itemize}
    We call $\sS$ the \emph{universal Kac--Moody action}.
\end{proposition}
We will view $\sS$ as a function on either $\cA^\omega$, $D^\omega_0G_\C$ or $L^\omega_1G$ depending on the context. In particular, we define
    \begin{equation}\label{eq:def-hat-S}
    \hat\sS(\gamma):=\sS(\gamma^{-1})\qquad\forall\gamma\in L^\omega_1G.
    \end{equation}

\begin{proof}
We will prove this by hand, using the single input of existence of directional derivatives in directions $u\in L^\omega\kg_\C$. 

Since we require $\sS\in\R$, we must also have $\cL_u\sS=-\overline{\ralpha(u^*)}$ for all $u\in D^\omega_0\kg_\C$. Uniqueness is then clear: if $\sS_1,\sS_2$ satisfy the required properties, then $\cL_u(\sS_1-\sS_2)=0$ for all $u\in L^\omega_0G_\C$. Now, let $\alpha=g^{-1}\del g\in\cA^\omega$ and consider the path $(g_r)_{r\in[0,1]}$ in $D^\omega_0G_\C$ defined by $g_r(z)=g(rz)$. We have also $\alpha_r(z)=rg_r^{-1}\del g_r(z)=zr\alpha(rz)$. Set $u_r:=\frac{\del}{\del\epsilon}_{|\epsilon=0}g_{re^{-\epsilon}}g_r^{-1}$; we have explicitly $u_r(z)=-rz\del gg^{-1}(rz)$ and note that $u_\cdot\in C^0([0,1],D^\omega_0\kg_\C)$. By the fundamental theorem of calculus, we have $\sS_1(\alpha)-\sS_2(\alpha)=\sS_1(0)-\sS_2(0)+\int_0^1(\cL_{u_r}\sS_1(\alpha_r)-\cL_{u_r}\sS_2(\alpha_r))\d r=0$. 

For the existence, recall that the Cauchy--Pompeiu formula gives a way to compute the $\del$-antiderivative of a $(1,0)$-form $\vartheta$ on a complex domain $D$ with $C^1$-boundary:
\[f(z):=\frac{1}{2i\pi}\oint_{\del D}\frac{\vartheta(\zeta)}{\zeta-z}-\frac{1}{2i\pi}\int_D\frac{\delbar\vartheta(\zeta)}{\zeta-z}.\]
Namely, $f$ is independent of the arbitrariness in the choice of $D$, and $\del f=\vartheta$. Now, suppose given an antiholomorphic embedding $q:\D\to\cA^\omega$ which is $C^1$ up to the boundary, and such that $q(0)=0$. Then, $q^*\ralpha$ is a $(1,0)$-form in $\D$, and $\delbar(q^*\ralpha)=q^*(\del\ralpha)=iq^*\omega_\km$. We can then apply the Cauchy--Pompeiu formula to $q^*\ralpha$, which only involves $\ralpha$ and $\omega_\km$ evaluated on the directions specified by $\delbar q\in C^0(\bar\D;\cA^\omega\simeq L^\omega\kg_\C)$. 

Hence, we can define unambiguously the $\delbar$-antiderivative of $\ralpha$ provided we can find such antiholomorphic embeddings $q$. Given $\ralpha=-\del gg^{-1}\in\cA^\omega$, we can associate a $\kg_\C$-valued $(0,1)$-form $\mu$ compactly supported in $\D^*$ (so that $\delbar gg^{-1}=\mu$). The map $\D\ni q\mapsto q\mu$ then defines such a holomorphic embedding of $\D$ into $\cA^\omega$ (with each $q\mu$ representing a unique point in $\cA^\omega$), which allows us to define $\sS$ unambiguously on $\cA^\omega$, and $\delbar\sS=-\ralpha$. Finally, the fact that $\sS$ is real is immediate from the fact that $\omega_\km$ itself is real. 
\end{proof}

\begin{remark}
    The definition of $\sS$ is by all means similar to the definition of the universal Liouville action \cite{Takhtajan-Teo06}, which is a K\"ahler potential for the Weil--Petersson on $\mathrm{Diff}^\omega(\S^1)/\mathrm{PLS}_2(\R)$. An important difference is that $\sS$ is not invariant under the inverse map ($\hat\sS\neq\sS$), contrary to the universal Liouville action. Probabilistically, this translates into the fact that $\nu_\kappa$ is not reversible. In fact, the very definition of the ``law of the inverse" under $\nu_\kappa$ is non-trivial, since the measure gives 0 mass to continuous loops. See Section~\ref{subsec:shapo} for a related discussion.
\end{remark}
\begin{remark}
    Given \cite[Theorem 1]{Takhtajan21_wzw}, we expect that $\sS$ has an expression involving the Wess--Zumino--Witten (WZW) action with target space $G_\C/G$. More precisely, if we view $\alpha\in\cA^\omega$ as representing a $G_\C$-bundle over $\hat\C$, and a $G_\C/G$-valued map on $\hat\C$ as a Hermitian metric in this bundle, then $\sS$ should realise the infimum of the $G_\C/G$-WZW action over all the Hermitian metrics in the bundle represented by $\alpha$. This identity will be the starting point of the coupling between $\nu_\kappa$ and the $G_\C/G$-WZW model. 
\end{remark}

The next lemma shows that $\sS\geq0$, a fact that we prove in a rather indirect way by relating the differential of $\hat\sS$ along the radial flow to the kinetic energy \eqref{eq:kinetic}, which is of independent interest.

\begin{lemma}\label{lem:positive}
    For all $\alpha\in\cA^\omega$, we have
    \[\lim_{t\to0}\frac{1}{t}\left(\hat\sS(e^{-t}\alpha(e^{-t}\cdot))-\hat\sS(\alpha)\right)=-\sE(\alpha)\leq0.\]
    As a consequence, $\sS\geq0$ on $\cA^\omega$ and $\sS(\alpha)=0$ if and only if $\alpha=0$.
\end{lemma}
\begin{proof}
    Let us write the Birkhoff factorisation $\gamma=g\hat g^{-1}\in L^\omega_1G$ with $\alpha=g^{-1}\del g$, and set $\hat\alpha:=\hat g^{-1}\del\hat g$ and $\hat\ralpha:=-\del\hat g\hat g^{-1}$. Setting $h_t(z):=g^{-1}(z)g(e^{-t}z)$ for any $t$ with $\Re(t)>0$, we have
    \[\del_t|_{t=0}h_t(z)=-z\alpha\qquad\text{and}\qquad\del_{\bar t}|_{t=0}h_t=0.\]
    The corresponding motion of the loop $\gamma$ is of the form $\gamma_t=\gamma\chi_t$ with $\chi_t=e^{tu+\bar tu^*+o(t)}$ with $u\in D^\omega_0\kg_\C$ is the solution to $\delbar u=-z\Ad_{\hat g}(\delbar\alpha)$ and $u(0)=0$ (where we have extended $\alpha$ smoothly to $\hat\C$ in some arbitrary way). By Stokes' formula and the fact that $\hat\ralpha=-\Ad_{\hat g}(\hat\alpha)$, we have
    \begin{align*}
    \oint_{\S^1}\tr(u(z)\hat\ralpha(z))\d z
    &=\int_{\D^*}\tr(\hat g(z)\del_{\bar z}\alpha(z)\hat g^{-1}(z)\hat\ralpha(z))z|\d z|^2\\
    &=-\int_{\D^*}\tr(\del_{\bar z}\alpha(z)\hat\ralpha(z))z|\d z|^2=\oint_{\S^1}\tr(\alpha(z)\hat\alpha(z))z\d z.
    \end{align*}
    On the other hand, we have $\cR_v(\hat\sS)=\frac{-1}{2i\pi}\oint\tr(v\hat\ralpha)$ for all $v\in D^\omega_0\kg_\C$ by definition of $\hat\sS$, hence    
    \begin{equation}\label{eq:expand-S-hat}
    \hat\sS(\alpha_t)=\hat\sS(\alpha)-2\Re\left(\frac{t}{2i\pi}\oint_{\S^1}\tr\left(\alpha(z)\hat\alpha(z)\right)z\d z\right)+o(t).
    \end{equation}
    It remains to show that what's inside the real part coincides with the kinetic energy. Using $\gamma^{-1}\del\gamma=\Ad_{\hat g}(\alpha+\hat\alpha)$ and the definition of the kinetic energy \eqref{eq:kinetic}, we find 
    \begin{align*}
    \sE(\gamma)=\frac{1}{2i\pi}\oint_{\S^1}\left(\tr\left(\alpha(z)\alpha(z)\right)+2\tr\left(\alpha(z)\hat\alpha(z)\right)+\tr\left(\hat\alpha(z)\hat\alpha(z)\right)\right)z\d z=\frac{1}{i\pi}\oint_{\S^1}\tr(\alpha(z)\hat\alpha(z))z\d z.
    \end{align*}
    The last equality just comes from the fact that $\tr(\alpha,\alpha)$ and $\tr(\hat\alpha,\hat\alpha)$ are holomorphic in $\D$ and $\D^*$ respectively. Plugging this into \eqref{eq:expand-S-hat}, we have $\hat\sS(\alpha_t)=\hat\sS(\alpha)-\Re(t\sE(\alpha))+o(t)$. Evaluating this on the real slice $\Im(t)=0$ gives that $\hat\sS$ is decreasing along the radial flow.

    It remains to prove that $\hat\sS\geq0$ with equality if and only if $\alpha=0$. The function $\hat\sS$ is of class $C^2$ on $\cA^\omega$, its differential at 0 vanishes, and its Hessian at 0 is positive definite (it is the Kac--Moody inner-product). Hence, the exists a neighbourhood $\cU$ of 0 in $\cA^\omega$ such that $\hat\sS>0$ in $\cU\setminus\{0\}$. Now, the radial flow starting from an arbitrary $\alpha\in\cA^\omega\setminus\{0\}$ converges to $0\in\cA^\omega$ as $t\to\infty$, so we can fix $t_0>0$ finite such that $\alpha_{t_0}\in\cU$, and obviously $\alpha_{t_0}\neq0$. We then have $\hat\sS(\alpha)\geq\hat\sS(\alpha_{t_0})>0$.
\end{proof}

A priori, the potential $\sS$ is only defined on $\cA^\omega$, but the next proposition shows that differences of potentials can be extended to the whole of $\cA$. More precisely, we establish that the covariance moduli $\Omega$ and $\Lambda$ appearing in Theorem \ref{thm:main} are well-defined continuous functions on $L^\omega_1G\times\cA$. Recall from Section \ref{subsec:action_complex} that we have a right action of $\hat D^\omega_\infty G_\C$ on $\cA$, and from Section~\ref{subsec:real_action} that we have a left action of $L^\omega_1G$ on $\cA$. We denote these actions respectively $(\chi,\alpha)\mapsto\chi\cdot\alpha$ and $(\alpha,h)\mapsto\alpha\cdot h$ if $\alpha\in\cA^\omega$, $h\in\hat D^\omega_\infty G_\C$ and $\chi\in L^\omega_1G$.
\begin{proposition}\label{prop:Omega}
    Let $\Omega:L^\omega_1G\times\cA^\omega\to\R$ and $\Lambda:\cA^\omega\times\hat D^\omega_\infty G_\C\to\R$ be defined by
    \[\Omega(\chi,\alpha):=\sS(\chi\cdot\alpha)-\sS(\alpha),\qquad\text{and}\qquad\Lambda(\alpha,h):=\sS(\alpha\cdot h)-\sS(\alpha).\]
    The functions $\Omega$ and $\Lambda$ extend continuously to functions on $L^\omega_1G\times\cA$ and $\cA\times\hat D^\omega_\infty G_\C$ respectively, which we still denote $\Omega$ and $\Lambda$.
\end{proposition}
\begin{proof}
We only treat the case of $\Omega$, the case of $\Lambda$ being identical.

    Fix $\chi\in L^\omega_1G$. Let $(\chi_t)_{t\in[0,1]}$ be a $C^1$-path in $L^\omega_1G$ such that $\chi_0=\ind_G$ and $\chi_1=\chi$. We denote $u_t:=\frac{\del}{\del\epsilon}_{|\epsilon=0}\chi_{t+\epsilon}\chi_t^{-1}$. By definition, we have $u_\cdot\in C^0([0,1];L^\omega\kg_\C)$. Hence, for all $\alpha\in\cA^\omega$, the function $t\mapsto\sS(\chi_t\cdot\alpha)$ is of class $C^1$ and we have
    \begin{equation}\label{eq:Omega}
    \Omega(\chi,\alpha)=\int_0^1\del_t\sS(\chi_t\cdot\alpha)\d t=\int_0^12\Re(\cL_{u_t}\sS(\chi_t\cdot\alpha))\d t.
    \end{equation}
    Writing $\alpha_t:=\chi_t\cdot\alpha$, we have $\cL_{u_t}\sS(\alpha_t)=\frac{1}{2i\pi}\oint\tr(u_t\alpha_t)$, where the contour can be chosen to be a circle in $\D$ uniformly bounded away from $\del\D$, e.g. $r\S^1$ for some fixed $r\in(0,1)$. Hence, for each $t\in[0,1]$, the map $\cA^\omega\ni\alpha\mapsto\oint\tr(u_t\alpha)$ extends continuously to a function on $\cA$. Since this is uniform in $t\in[0,1]$, \eqref{eq:Omega} shows that $\Omega(\chi,\cdot)$ extends continuously to $\cA$. Finally, \eqref{eq:Omega} is clearly continuous in $\chi$, so that the extension is in $C^0(L^\omega_1G\times\cA)$.
    \end{proof}

\section{Proof of Theorem \ref{thm:main}}\label{sec:proof}

We endow $\cA$ with the local uniform topology, i.e. the topology generated by the family of seminorms 
 \[\left\lbrace\Vert\alpha\Vert_{C^0(K)}:=\sup_{z\in K}|z\alpha(z)|\,\big|\,K\subset\D\text{ compact}\right\rbrace,\]
 where we have defined $|u|:=\sqrt{-\tr(u\bar u)}$ for any $u\in\kg_\C$. The space $\cA$ is separable and completely metrisable, e.g. with the distance $d_\cA(\alpha_1,\alpha_2)=\sum_{n=1}^\infty\frac{2^{-n}\Vert\alpha_1-\alpha_2\Vert_{C^0(e^{-1/n}\bar\D)}}{1+\Vert\alpha_1-\alpha_2\Vert_{C^0(e^{-1/n}\bar\D)}}$. This turns $\cA$ into a Polish space, admitting $\cA^\omega$ as a dense subspace. By Montel's theorem, for every $R>0$, the set
\begin{equation}\label{eq:def-KR}
    K_R:=\left\lbrace\alpha\in\cA|\,\forall n\in\N_{>0},\,\Vert\alpha\Vert_{C^0(e^{-1/n}\bar\D)}\leq Rn^4\right\rbrace
\end{equation}
is a compact subset of $\cA$ (the arbitrary exponent 4 is chosen to make some series converge later). We will denote by 
\begin{equation}\label{eq:Borel}
\P(\cA):=\{\text{Borel probability measures on }\cA\}.
\end{equation}
Since $\cA$ is Polish, every $\nu\in\P(\cA)$ is inner-regular. The space $\cA$ is neither compact nor locally compact, but Lemma \ref{lem:compact} shows a priori that every unitarising measure must give full mass to $\cup_{R>0}K_R$. Thus, as far as unitarising measures are concerned, the space $\cA$ looks like a locally compact space exhausted by the family $(K_R)_{R>0}$. We will say that a measure $\nu\in\P(\cA)$ is \emph{characterised by its moments} if the following two conditions hold: first, $\cC^{1,0}\subset L^2(\nu)$; second, if $\tilde\nu\in\P(\cA)$ is such that $\cC^{1,0}\subset L^2(\tilde\nu)$ and $\langle P,Q\rangle_{L^2(\tilde\nu)}=\langle P,Q\rangle_{L^2(\nu)}$ for all $P,Q\in\cC^{1,0}$, then $\tilde\nu=\nu$ in $\P(\cA)$.

Until we are able to unite them in the same object (i.e. at the end of Section \ref{subsec:mgf}), we will call a measure satisfying \eqref{eq:quasi-invariance} an \emph{$L^\omega_1G$-covariant measure}, and a measure satisfying \eqref{eq:quasi-invariance-bis} a \emph{$\hat D_\infty^\omega G_\C$-covariant measure}. 
 
The plan of the proof is as follows. In Section \ref{subsec:uniqueness}, we prove that $L^\omega_1G$-covariant measures are also $\hat D^\omega_\infty G_\C$-covariant. In Section~\ref{subsec:existence}, we construct an $L^\omega_1G$-covariant measure. In Section~\ref{subsec:mgf}, we prove that this measure is characterised by its moments, and that any $L^\omega_1G$-covariant measure has its moments determined, completing the proof of the theorem. Finally, Section \ref{subsec:shapo} discusses the closures of the operators $(\bJ_u,\bar\bJ_u,\bL_u,\bR_u)_{u\in L^\omega\kg_\C}$ on $L^2(\nu_\kappa)$ and the Shapovalov forms of the corresponding Kac--Moody representations. The levels will always be related by the formula $\check\kappa=-2\check h-\kappa>0$.

\begin{remark}\label{rem:abuse}
Throughout the proof, we will frequently write integrals involving expressions like $\cL_u(\sS)$ or $\cJ_u(\sS)$, by which we mean the continuous extension of these expressions from $\cA^\omega$ to $\cA$. For instance, the function $\cJ_u\sS(\alpha)=\alpha(u)$ has a continuous extension to $\cA$ even though $\sS=\infty$ with full $\nu_\kappa$-probability. This slight abuse of notation makes the computations easier to follow.
\end{remark}

    \subsection{\texorpdfstring{$L^\omega_1G$}{LG}-covariant measures are \texorpdfstring{$\hat D^\omega_\infty G_\C$}{DG}-covariant}\label{subsec:uniqueness}

The goal of this section is to exhibit the link between $L^\omega_1G$-covariant and $\hat D^\omega_\infty G_\C$-covariant measures, which is perhaps the most striking aspect of Theorem \ref{thm:main}.

\begin{proposition}\label{prop:moments}
Let $\nu\in\P(\cA)$ be an $L^\omega_1G$-covariant measure. Then, $\nu$ is $\hat D^\omega_\infty G_\C$-covariant.
\end{proposition}

We postpone the proof to the end of the section, requiring the next two lemmas beforehand. We start with an a priori bound on the support of $L^\omega_1G$-covariant measures. Recall the definition of $K_R$ in \eqref{eq:def-KR}.

\begin{lemma}\label{lem:compact}
    Let $\nu\in\P(\cA)$ be an $L^\omega_1G$-covariant measure. Then, $\nu(\cup_{R>0}K_R)=1$.
\end{lemma}
\begin{proof}
     Let $\rho\in C^\infty_c(\cA)$ such that $\rho\equiv1$ on $\cup_{R>0}K_R$. 
     
     By differentiating the covariance formula twice at $\chi=\ind_{G_\C}$ and using that $\bL_{b,m}(\bL_{b,-m}(\ind))=\check\kappa^2|\ralpha_{b,m}|^2-\check\kappa m$, we get that $\int_\cA|\ralpha_{b,m}|^2\rho\d\nu=O(m)$. Hence, by Markov's inequality and the Borel--Cantelli lemma, for $\nu$-a.e. $\ralpha\in\mathrm{supp}(\rho)$, there is $m_0\in\N_{>0}$ such that $|\ralpha_{b,m}|\leq m^{3/2}$ for all $b\in\kB, m\geq m_0$. Now, for each $r\in(0,1)$, we have almost surely by Cauchy--Schwarz $\sup_{z\in r\bar\D}|z\ralpha(z)|\leq(\sum_{b\in\kB,m\geq1}r^m\sum_{b\in\kB,m\geq1}|\ralpha_{b,m}|^2r^m)^{1/2}=O_{r\nearrow1}((1-r)^{-5/2})$. This readily implies that $\ralpha\in K_R$ for some $R>0$. 

     Thus, we have shown that $\nu(K)=\nu(K\cap(\cup_{R>0}K_R))$ for all compact sets $K\subset\cA$. This concludes the proof since $\nu$ is inner-regular as a Borel measure on a Polish space. 
\end{proof}

We will now prove the infinitesimal version of Proposition \ref{prop:moments}, namely an integration by parts formula satisfied by any $L^\omega_1G$-covariant measure. 

\begin{lemma}\label{lem:induction}
    Let $\nu\in\P(\cA)$ be an $L^\omega_1G$-covariant measure. We have
    \begin{equation}
\int_\cA\bJ_u(F)\d\nu=0=\int_\cA\bar\bJ_u(F)\d\nu,\qquad\forall F\in C^\infty_c(\cA),\forall u\in\hat D^\omega_\infty\kg_\C.
    \end{equation}
\end{lemma}

\begin{proof}
 For all $u\in L^\omega\kg_\C$, applying \eqref{eq:quasi-invariance} and dominated convergence to $\chi=e^{tu+\bar tu^*}$ and differentiating at $t=0$ gives for all $F\in C^\infty_c(\cA)$ (recall Remark \ref{rem:abuse})
 \begin{equation}\label{eq:ibp_L}
 \int_\cA\cL_u(F)\d\nu=\check\kappa\int_\cA\cL_u(\sS)F\d\nu.
 \end{equation}
By definition of the representation $(\bL_u)_{u\in L^\omega\kg_\C}$, this implies
$\int_\cA\bL_u(F)\d\nu=0$ for all $u\in\hat D^\omega_\infty\kg_\C$. The core of the proof is to understand the change of coordinate from $(\cL_u)_{u\in L^\omega_0\kg_\C}$ to $(\cJ_u,\bar\cJ_u)_{u\in\hat D^\omega_\infty\kg_\C}$, which can be understood as the Jacobian between the (non-existent) Haar measures on $LG$ and $D_0G_\C$, namely we have heuristically ``$D\hat g=e^{2\check h\sS(\gamma)}D\gamma$" under the usual parametrisation $\gamma=g\hat g^{-1}\in L^\omega G$.

     Let $u\in\hat D^\omega_\infty\kg_\C$ and $\alpha=g^{-1}\del g\in\cA$. Extend $u$ smoothly to $\hat\C$ in an arbitrary way. We consider the small motion in $g_t=e^{-t\tilde u-\bar t\tilde u^*+o(t)}ge^{tu}\in D_0G_\C$ where $\tilde u$ is the solution to $\delbar\tilde u=\Ad_g(\delbar u)$ and $\tilde u(\infty)=0$. Explicitly, using the Cauchy transform \eqref{eq:cauchy_transform} and expanding at $z=\infty$, we have
\begin{equation}\label{eq:def_beta}
\tilde u(z)=\sum_{b\in\kB,m\leq-1}\beta_{b,m}^ub\otimes z^m\qquad\text{with}\qquad\beta^u_m=\frac{1}{2i\pi}\oint g(z)u(z)g^{-1}(z)\frac{\d z}{z^{m+1}},
\end{equation}
and $\beta_{b,m}^u=\tr(\beta_m^u,b)$. Here, the contour can be chosen to be any circle in $\D$ containing the support of $\delbar u$ in its interior. 

By definition, $\cJ_u$ is the differential along infinitesimal motion generated by $u$. On the other hand, the expression $g_t=e^{-t\tilde u-\bar t\tilde u^*}ge^{tu}$ is also the small motion generated by $\tilde u$ along the action introduced in Section \ref{subsec:real_action}. Hence, when acting on test functions, we have $\cJ_u=\cL_{\tilde u}=\sum_{b\in\kB,m\leq-1}\beta^u_{b,m}\cL_{b,m}$. By \eqref{eq:ibp_L}, we have for all $F\in C^\infty_c(\cA)$ and all $M\geq1$,
\begin{equation}\label{eq:some_ibp}
\int_\cA\sum_{b\in\kB,-M\leq m\leq -1}\beta^u_{b,m}\cL_{b,m}(F)\d\nu=\int\sum_{b\in\kB,-M\leq m\leq-1}\left(\check\kappa\beta_{b,m}^u\cL_{b,m}(\sS)-\cL_{b,m}(\beta^u_{b,m})\right)F\d\nu.
\end{equation}
Since $F$ is a test function, we have the convergence $\sum_{b\in\kB,-M\leq m\leq-1}\beta^u_{b,m}\cL_{b,m}(F)\to\cJ_u(F)$ in $L^1(\nu)$. Moreover, for all $\alpha\in\cA$, we have (recall the basis $b_m=b\otimes z^m$)
\[\sum_{b\in\kB,-M\leq m\leq-1}\beta_{b,m}^u\cL_{b,m}(\sS)
=-\sum_{b\in\kB,-M\leq m\leq-1}\ralpha_{b,m}(\beta^u_{b,m}b_m)\underset{M\to\infty}\to-\ralpha(\tilde u).\]
By Stokes' formula and the fact that $\ralpha=-\Ad_g(\alpha)$, we also have
\[\ralpha(\tilde u)=\frac{i}{2\pi}\int_\D\tr(\ralpha\wedge\delbar\tilde u)=\frac{1}{2i\pi}\int_\D\tr(\alpha\wedge\delbar u)=-\alpha(u).\]
By dominated convergence, the convergence $\sum_{b\in\kB,m\leq-1}\beta_{b,m}^u\cL_{b,m}(\sS)=\alpha(u)$ then also holds in $L^1(\nu)$. Hence, \eqref{eq:some_ibp} can be rewritten as
\begin{equation}\label{eq:another_ibp}
\int_\cA\left(\cJ_u(F)-\check\kappa\alpha(u)F\right)\d\nu+\int_\cA\sum_{b\in\kB,m\leq-1}\cL_{b,m}(\beta^u_{b,m})F\d\nu=0.
\end{equation}
The result follows upon showing that (recall $-\check\kappa=2\check h+\kappa$) 
\begin{equation}\label{eq:ghost}
\sum_{b\in\kB,m\leq-1}\cL_{b,m}(\beta_{b,m}^u)=-2\check h\alpha(u)\text{ in }L^1(\nu).
\end{equation}

To prove this identity, we start by fixing an arbitrary $v\in\hat D^\omega_\infty\kg_\C$ and find a formula for $\cL_v(\beta_{b,m}^u)$. We extend $v$ to the whole sphere in an arbitrary way and note that $\delbar v$ is compactly supported in $\D$. The deformation induced by $v$ is of the form $g_t=e^{-tv-\bar tv^*}ge^{t\tilde v+o(t)}$ with $\tilde v$  solving $\delbar\tilde v=\Ad_{g^{-1}}(\delbar v)$ and $\tilde v(\infty)=0$. Explicitly, for all $z$ away from the support of $\delbar v$, we have
\[\tilde v(z)=\frac{1}{2i\pi}\oint g^{-1}(\zeta)v(\zeta)g(\zeta)\frac{\zeta\d\zeta}{z(\zeta-z)},\]
where the contour can be taken to be a circle inside $\D$ containing the support of $\delbar u$ in its interior and $z$ in its exterior. Extracting the $t$-derivative at $t=0$ in the definition of $\beta_m^u$ \eqref{eq:def_beta}, we get: 
\[\cL_v(\beta_m^u)=\frac{-1}{2i\pi}\oint g(z)[u(z),\tilde v(z)]g^{-1}(z)\frac{\d z}{z^{m+1}}+\frac{1}{2i\pi}\oint[g(z)u(z)g^{-1}(z),v(z)]\frac{\d z}{z^{m+1}}.\]
Now, we apply this to $v=b_m=b\otimes z^m$ in order to compute $\cL_{b,m}(\beta_{b,m}^u)=\tr(\cL_{b,m}(\beta_m^u),b)$. By invariance of the Killing form \eqref{eq:killing}, we see that the second term in the right-hand-side of the previous equation vanishes. As for the first term, the expression of $\tilde u$ with the Cauchy transform and the invariance of the Killing form yield: 
\[\cL_{b,m}(\beta_{b,m}^u)=\frac{1}{4\pi^2}\oint\oint\tr(u(z),[g(\zeta)bg^{-1}(\zeta),g(z)bg^{-1}(z)])\frac{\zeta^{m+1}\d\zeta\d z}{z^{m+2}(\zeta-z)}.\]
In this formula, the contours can be chosen to be circles inside $\D$ containing the support of $\delbar u$ in their interior, with the $\zeta$-radius strictly smaller than the $z$-radius. Note however that the apparent singularity at $z=\zeta$ is removable (since we get the commutator of $g(z)bg^{-1}(z)$ with itself, which vanishes), so the formula is unchanged if we take instead the $\zeta$-radius to be strictly greater than the $z$-radius. We can then sum the geometric series over $m\leq-1$ and get for all $b\in\kB$:
\[\sum_{m\leq-1}\cL_{b,m}(\beta_{b,m}^u)=\frac{-1}{4\pi^2}\oint\oint\tr\left(g(z)u(z)g^{-1}(z),[b,g(z)g^{-1}(\zeta)bg(\zeta)g^{-1}(z)]\right)\frac{\d\zeta\d z}{(\zeta-z)^2}\]
For fixed $z$, the $\zeta$ integrand is holomorphic inside the unit disc with a unique double pole at $z=\zeta$, so it reduces to the differential evaluated at $\zeta=z$. Using the elementary computation $\frac{\del}{\del\zeta}_{|\zeta=z}g(z)g^{-1}(\zeta)bg(\zeta)g^{-1}(z)=[b,\del gg^{-1}(z)]$, the previous displayed equation reduces to
\[\sum_{m\leq-1}\cL_{b,m}(\beta_{b,m}^u)=\frac{-1}{2i\pi}\oint\tr\left(g(z)u(z)g^{-1}(z),\left[b,[b,\del gg^{-1}(z)]\right]\right)\d z\]
 Summing over $b\in\kB$, we recognise the Casimir in the adjoint representation \eqref{eq:casimir}, so \eqref{eq:coxeter} gives
\begin{align*}
\sum_{b\in\kB,m\leq-1}\cL_{b,m}(\beta_{b,m}^u)
&=\frac{-2\check h}{2i\pi}\oint\tr(g^{-1}(z)u(z)g(z),\del gg^{-1}(z))\d z=\frac{-2\check h}{2i\pi}\oint\tr(ug^{-1}\del g).
\end{align*}
This proves the convergence \eqref{eq:ghost} for $\nu$-a.e. $\alpha\in\cA$. In fact, using the earlier expression as a geometric series and the dominated convergence theorem, this convergence is easily promoted to a convergence in $L^1(\nu)$. With this input, \eqref{eq:another_ibp} becomes $\int_\cA\bJ_u(F)\d\nu=\int_\cA(\cJ_u(F)+\kappa\alpha(u)F)\d\nu=0$. The proof that $\int_\cA\bar\bJ_u(F)\d\nu=0$ is identical, ending the proof of the lemma. 
 \end{proof}

\begin{proof}[Proof of Proposition \ref{prop:moments}]
We are going to integrate the infinitesimal formula of Lemma \ref{lem:induction}.

 Let $h\in\hat D^\omega_\infty G_\C$ and $(h_t)_{t\in[0,1]}$ be a $C^1$-family in $\hat D^\omega_\infty G_\C$ with $h_0=\ind_G$ and $h_1=h$. We write $u_t:=\frac{\del}{\del\epsilon}_{|\epsilon=0}h_t^{-1}h_{t+\epsilon}$, and note that $u_\cdot\in C^0([0,1];\hat D^\omega_\infty\kg_\C)$.

Let $F\in C^\infty_c(\cA)$, and define $F_{h_t}(\alpha):=F(\alpha\cdot h_t)$ for all $t\in[0,1]$ and all $\alpha\in\cA$. By Lemma \ref{lem:induction} and the definition of $\Lambda$ from Proposition \ref{prop:Omega}, we have
    \begin{align*}
    \del_t\int_\cA F(\alpha\cdot h_t)e^{\kappa\Lambda(\alpha,h_t)}\d\nu(\alpha)
    &=\int_\cA\left(\cJ_{u_t}F_{h_t}(\alpha)+\kappa F_{h_t}(\alpha)\cJ_{u_t}\sS_{h_t}(\alpha)\right)e^{\kappa\Lambda(h_t,\alpha)}\d\nu(\alpha)\\
    &\quad+\int_\cA\left(\bar\cJ_{u_t}F_{h_t}(\alpha)+\kappa F_{h_t}(\alpha)\bar\cJ_{u_t}\sS_{h_t}(\alpha)\right)e^{\kappa\Lambda(h_t,\alpha)}\d\nu(\alpha)\\
&=\int_\cA\cJ_{u_t}\left(F_{h_t}e^{\kappa\Lambda(h_t,\cdot)}\right)\d\nu+\kappa\int_\cA \cJ_{u_t}(\sS)F_{h_t}e^{\kappa\Lambda(\cdot,h_t)}\d\nu\\
&\quad+\int_\cA\bar\cJ_{u_t}\left(F_{h_t}e^{\kappa\Lambda(h_t,\cdot)}\right)\d\nu+\kappa\int_\cA\bar\cJ_{u_t}(\sS)F_{h_t}e^{\kappa\Lambda(\cdot,h_t)}\d\nu=0.
    \end{align*}
    Integrating in $t$ gives $\int_\cA F(\alpha\cdot h)e^{\kappa\Lambda(\alpha,h)}\d\nu(\alpha)=\int_\cA F(\alpha)e^{\kappa\Lambda(\alpha,\ind_{G_\C})}\d\nu(\alpha)=\int_\cA F(\alpha)\d\nu(\alpha)$, which gives the result for test functions. By Lemma \ref{lem:compact}, in order to extend this to all $F\in C^0_b(\cA)$, it suffices to consider functions supported in $\cup_{R>0}K_R$.
    
    Let then $\rho\in C^\infty_c(\cA)$ be such that $\rho$ is supported in $K_2$ and $\rho\equiv1$ on $K_1$. Let also $\rho_R(\alpha):=\rho(R^{-1}\alpha)$ for all $\alpha\in\cA$, $R>0$ (so that $\rho_R$ is supported in $K_{2R}$). By dominated convergence, for all $\chi\in L^\omega_1G$, we have $\int_\cA\rho_R(\chi^{-1}\cdot\alpha)\d\nu\to1$ as $R\to\infty$. It then follows Fatou's lemma and the covariance formula for test functions that $e^{\kappa\Lambda(\cdot,h)}\in L^1(\nu)$ and $\int_\cA e^{\kappa\Lambda(\alpha,h)}\d\nu=\lim_{R\to\infty}e^{\kappa\Lambda(\cdot,h)}\rho_R\d\nu=1$. The fact that $\nu$ is a $\hat D^\omega_\infty G_\C$-covariant measure then follows from the density of $C^\infty_b(\cA)$ in $C^0_b(\cA)$.
\end{proof}

    \subsection{Existence of a unitarising measure}\label{subsec:existence}

In this section, we prove the existence of an $L^\omega_1G$-covariant measure, which we state as the proposition below.

\begin{proposition}\label{prop:existence}
There exists an $L^\omega_1G$-covariant measure at level $\check\kappa$ for every $\check\kappa>0$.
\end{proposition}

Our strategy will be to construct a unitarising measure as the weak limit of a family of weighted Wiener measures. The reason for this choice (compared to, say, a Gaussian measure on $\cA$) is that Wiener measures transform covariantly, so we will be able to pass the integration by parts formulas to the limit. Before going into the proof, we give some preliminary background on these measures. The main reference is \cite{Malliavins}, especially Theorem B therein.

Recall the definition of the kinetic energy $\sE$ from \eqref{eq:kinetic}. Given $\chi\in L^\omega G$, an elementary computation gives
\[\sE(\chi\gamma)=\sE(\gamma)+\frac{i}{\pi}\oint_{\S^1}\tr(\chi^{-1}\del_z\chi,\del_z\gamma\gamma^{-1})z\d z+\sE(\chi).\]
Now, suppose $\chi_t=e^{-tu-\bar tu^*}$ for some $u\in D^\omega_0\kg_\C$. Differentiating at $t=0$, the previous formula gives 
\begin{equation}\label{eq:diff_E}
\cL_u\sE(\gamma)=\frac{i}{\pi}\oint_{\S^1}\tr(\del_zu,\del_z\gamma\gamma^{-1})z\d z.
\end{equation}
Given $\sigma>0$, we denote by $\P_\sigma$ the Wiener measure with variance $\sigma^{-2}$ (and $\E_\sigma$ its expectation) on the loop space $C^0_1(\S^1;G):=\{\gamma\in C^0(\S^1;G)|\,\gamma(1)=1_G\}$. 
For us, the most important property of the Wiener measure will be the following integration by parts formula \cite[Theorem B]{Malliavins}. For all $u\in D^\omega_0\kg_\C$ and all $F\in C^\infty_c(\cA)$, 
\begin{equation}\label{eq:ibp_BM}
\E_\sigma\left[\cL_u(F)\right]=\sigma^2\E\left[\cL_u(\sE)F\right].
\end{equation}
We remind the reader of Remark \ref{rem:abuse} for the abuse of notation in this formula.

Having the regularity of Brownian motion, the Wiener measure gives full mass to the H\"older space $C^s(\S^1;G)$ for any $p\in(0,\frac{1}{2})$. Moreover, according to \cite[Theorem~6.2]{Clancey-Gohberg}, every $\gamma\in C^s(\S^1;G)$ admits a Birkhoff factorisation $\gamma=g\lambda\hat g^{-1}$ where $g\in D_0G_\C$ extends $s$-H\"older continuously to $\bar\D$.\footnote{The space $C^s(\S^1;\mathrm{GL}_n)$ corresponds to $G[H_s(\S^1)]_{n,n}$ in the notation from \cite[Theorem 6.2]{Clancey-Gohberg}. We then apply the result to any faithful representation of $G$.} This factorisation is of course unique modulo the global gauge symmetry, so we have a well-defined map $C^s(\S^1;G)\to\cA$. In fact, this map is continuous since the argument of \cite{Clancey-Gohberg} is based on a passage to the limit thanks to the boundedness of an operator on the relevant H\"older spaces. As a result, we have a measurable map $C^0(\S^1;G)\to\cA$ defined $\P_\sigma$-a.e., and we denote by $\tilde\P_\sigma$ the corresponding pushforward on $\cA$. Otherwise stated, $\tilde\P_\sigma$ is the marginal distribution of $g^{-1}\del g\in\cA$ when writing a Wiener-distributed loop $\gamma=g\lambda\hat g^{-1}$ according to its Birkhoff factorisation. We also denote by $\tilde\P^\epsilon_\sigma$ the law of $e^{-\epsilon}\alpha(e^{-\epsilon}\cdot)$ when $\alpha$ is sampled from $\tilde\P_\sigma$.

\begin{proof}[Proof of Proposition \ref{prop:existence}]
 With the setup above, for each $\sigma>0$, we define the probability measure
\begin{equation}\label{eq:def-nu-sigma}
\d\nu^\sigma_\kappa:=\frac{1}{\mathcal{Z}_\kappa^\sigma}e^{-\check\kappa\sS}\d\tilde\P^\sigma_\sigma\qquad\text{with}\qquad\mathcal{Z}_\kappa^\sigma:=\int_\cA e^{-\check\kappa\sS}\d\tilde\P^\sigma_\sigma.
\end{equation}
This is well-defined since $\sS$ is finite $\tilde\P_\sigma^\sigma$-a.e. (recall that samples of $\tilde\P_\sigma^\sigma$ extend holomorphically to $e^\sigma\D$) and $e^{-\check\kappa\sS}\leq1$ by Lemma \ref{lem:positive}. We seek to prove that $\nu^\sigma_\kappa$ converges weakly to a unitarising measure as $\sigma\to0$. By the quasi-invariance of the Wiener measure \cite[Theorem~B]{Malliavins}, the moment generating function $\mathfrak F_{\nu^\sigma_\kappa}$ converges everywhere, for all $\sigma>0$; in particular, all the moments are finite. In the remainder of the proof, we will first argue that the sequence $(\nu^\sigma_\kappa)_{\sigma>0}$ is tight in $\P(\cA)$ as $\sigma\to0$ (via a second moment computation), then show that any subsequential limit must be a unitarising measure. 

We start by computing the second moments of the Fourier coefficients of $\del\gamma\gamma^{-1}$ under the Wiener measure $\P_\sigma$. Using \eqref{eq:diff_E} and \eqref{eq:ibp_BM}, we have $\E_\sigma\left[|\cL_u\sE|^2\right]=\sigma^{-2}\E_\sigma\left[\cL_{u^*}(\cL_u(\sE))\right]$. Moreover, it is straightforward to compute $\cL_v(\cL_u(\sE))=\frac{1}{i\pi}\oint_{\S^1}(\tr(\del_zu,\del_zv)+\tr(\del_z\gamma\gamma^{-1},[\del_zu,\del_zv]))z\d z$ for all $u,v\in L^\omega\kg_\C$. Applying this to $u=v^*=b\otimes z^m$ for some $b\in\kB$, $m\geq1$, we finally get
\begin{equation}
    \E_\sigma[|\cL_{b,m}(\sE)|^2]=2\sigma^{-2}m^2.
\end{equation}

 We continue with the second moments of the coefficients $\ralpha_{b,m}$ under $\nu^\sigma_\kappa$. Using that $\bL_{b,-m}\ind=\check\kappa\ralpha_{b,m}$ and $\cL_{b,m}(\ralpha_{b,m})=m\check\kappa$, the integration by parts formula \eqref{eq:ibp_BM} gives
 \begin{equation}\label{eq:second-moment}
 \int_\cA(\check\kappa^2|\ralpha_{b,m}|^2-\check\kappa m)\d\nu_\kappa^\sigma=-\sigma^2\int_\cA\cL_{b,m}(\sE^{-\sigma})\d\nu^\sigma_\kappa.
 \end{equation}
 Here, if $\gamma=g\hat g^{-1}$ is such that $g$ has a holomorphic continuation to $e^\sigma\D$ (which is $\nu^\sigma_\kappa$-a.s. the case), we have defined $\sE^{-\sigma}(\gamma)$ to be the kinetic energy of the loop represented by $e^\sigma g(e^\sigma\cdot)$. Under $\nu^\sigma_\kappa$, the tail of the coefficient $\cL_{b,m}(\sE^{-\sigma})$ is lighter than that of $\cL_{b,m}(\sE)$ under $\P_\sigma$, hence $\int_\cA|\cL_{b,m}(\sE^{-\sigma})|\d\nu^\sigma_\kappa\leq(\E_\sigma[|\cL_{b,m}(\sE)|^2])^{1/2}\leq\sqrt{2}\sigma^{-1}m^2$. Hence, \eqref{eq:second-moment} can be rewritten
 \begin{equation}\label{eq:moment-two-limit}
 \int_\cA|\ralpha_{b,m}|^2\d\nu^\sigma_\kappa=\frac{m}{\check\kappa}+m^2O(\sigma).
 \end{equation}

 
 This estimate readily implies that $(\nu^\sigma_\kappa)_{\sigma>0}$ is tight in the space of Borel probability measures on $\cA$. Indeed, for each $r\in(0,1)$ and all $\ralpha\in\cA$, we have by the Cauchy--Schwarz inequality $\sup_{z\in r\bar\D}|z\alpha(z)|^2\leq\frac{\mathfrak c_2}{1-r}\sum_{b\in\kB,m\geq1}r^m|\ralpha_{b,m}|^2$, for some constant $\mathfrak c_2>0$ independent of everything. 
Together with \eqref{eq:moment-two-limit}, we see that the $\nu^\sigma_\kappa$-expectation of $\sup_{z\in r\bar\D}|z\alpha(z)|$ is $O((1-r)^{-2})$ as $r\to1$, independently of $\sigma$. By Markov's inequality and a union bound, we then have for all $\sigma,R>0$ and some constant $\mathfrak c>0$ independent of $\sigma$ (recall \eqref{eq:def-KR} for the definition of $K_R$)
\[\nu^\sigma_\kappa(\cA\setminus K_R)\leq\sum_{n=1}^\infty\nu_\kappa^\sigma(\{\Vert\ralpha\Vert_{C^0(e^{-1/n}\bar\D)}\geq Rn^4\})\leq\frac{\mathfrak c}R\sum_{n=1}^\infty n^{2-4}=O(R^{-1}).\]
This concludes the proof that the sequence $(\nu_\kappa^\sigma)_{\sigma>0}$ is tight in $\P(\cA)$ as $\sigma\to0$. Let $\cM$ be its set of accumulation points ($\cM\neq\emptyset$ by Prokhorov's theorem). 

Let $\nu\in\cM$ and let us show that $\nu$ is $L^\omega_1G$-covariant. For all $\sigma>0$, $u\in L^\omega\kg_\C$ and $F\in C^\infty_c(\cA)$, we have by \eqref{eq:ibp_BM}
    \[\int_\cA\left(\cL_u(F)-\check\kappa\cL_u(\sS)F\right)\d\nu^\sigma_\kappa=\sigma^2\int_\cA\cL_u(\sE^{-\sigma})F\d\nu^\sigma_\kappa\to0.\]
    We have already seen that the right-hand-side converges to 0 as $\sigma\to0$. Moreover, the left-hand-side converges to $\int_\cA\bL_u(F)\d\nu$ along any subsequence $\sigma_n\searrow0$ such that $\nu^{\sigma_n}_\kappa\to\nu$ weakly. Hence, 
    \begin{equation}\label{eq:ibp_existence}
    \int_\cA\left(\cL_u(F)-\check\kappa\cL_u(\sS)F\right)\d\nu=0,\qquad\forall u\in L^\omega\kg_\C. 
    \end{equation}
    
   We may now integrate this relation to find that $\nu$ satisfies the unitarising property \eqref{eq:quasi-invariance} for all test functions $F\in C^\infty_c(\cA)$. Let $\chi\in L^\omega_1G$ and fix a $C^1$-path $(\chi_t)_{t\in[0,1]}$ in $L^\omega_1G$ such that $\chi_0=\ind_G$ and $\chi_1=\chi$. Denote $u_t:=\frac{\del}{\del\epsilon}_{|\epsilon=0}\chi_t^{-1}\chi_{t+\epsilon}$. Let $F\in C^\infty_c(\cA)$ and denote $F_{\chi_t}(\alpha):=F(\chi_t\cdot\alpha)$. From \eqref{eq:ibp_existence}, we have
    \begin{align*}
    \del_t\int_\cA F(\chi_t\cdot\alpha)e^{-\check\kappa\Omega(\chi_t,\alpha)}\d\nu(\alpha)
    &=\int_\cA\left(\cL_{u_t}F_{\chi_t}(\alpha)-\check\kappa F_{\chi_t}(\alpha)\cL_{u_t}\sS_{\chi_t}(\alpha)\right)e^{-\check\kappa\Omega(\chi_t,\alpha)}\d\nu(\alpha)\\
    &=\int_\cA\cL_{u_t}\left(F_{\chi_t}e^{-\check\kappa\Omega(\chi_t,\cdot)}\right)\d\nu-\check\kappa\int_\cA\cL_{u_t}(\sS)F_{\chi_t}e^{-\check\kappa\Omega(\chi_t,\cdot)}\d\nu=0.
    \end{align*}
    Integrating in $t$ gives $\int_\cA F(\chi\cdot\alpha)e^{-\check\kappa\Omega(\chi,\alpha)}\d\nu(\alpha)=\int_\cA F(\alpha)e^{-\check\kappa\Omega(\ind_G,\alpha)}\d\nu(\alpha)=\int_\cA F(\alpha)\d\nu(\alpha)$, which gives the result for test functions. We conclude that the result holds for all $F\in C^0_b(\cA)$ using plateau functions and a density argument as in the end of the proof of Proposition \ref{prop:moments}.
\end{proof}

    \subsection{Uniqueness of \texorpdfstring{$L^\omega_1G$}{LG}-covariant measures}\label{subsec:mgf}

In this section, we complete the proof of Theorem \ref{thm:main} in two steps.  In Proposition \ref{lem:moments}, we prove that every $\hat D^\omega_\infty G_\C$-covariant measure has finite moments and that these moments are determined by \eqref{eq:quasi-invariance-bis}, implying the same thing for $L^\omega_1G$-covariant measures by Proposition \ref{prop:moments}. In Proposition \ref{prop:uniqueness}, we prove that the $L^\omega_1G$-covariant measure constructed in Proposition \ref{prop:existence} is characterised by its moments. Combining these two inputs ends the proof of Theorem \ref{thm:main}.

\begin{proposition}\label{lem:moments}
Let $\nu\in\P(\cA)$ be a $\hat D^\omega_\infty G_\C$-covariant measure. Then, $\cC^{1,0}\subset L^2(\nu)$, and the moments are given by
\begin{equation}\label{eq:moments}
\int_\cA\Psi_{\bk,\tilde\bk}\d\nu=\delta_{\bk,\emptyset}\delta_{\tilde\bk,\emptyset}\qquad\forall\bk,\tilde\bk\in\cT.
    \end{equation}
\end{proposition}

\begin{proof}
Arguing as in Lemma \ref{lem:compact}, it is easy to see that $\nu(\cup_{R>0}K_R)=1$ and we omit the details. 

     We will now prove the property by induction on the level of the polynomial. For the initialisation, we note that the constant function $\ind_G$ is integrable since $\nu$ is a probability measure, and in particular $\int_\cA\Psi_{\emptyset,\emptyset}\d\nu=1$. 

    Suppose $\cC^{1,0}_{N-1}\subset L^2(\nu)$ for some $N\in\N_{>0}$, and note that this implies $\oplus_{0\leq m\leq 2N-2}\cC^{1,0}_n\otimes\cC^{0,1}_{2N-2-m}\subset L^1(\nu)$. Let $m\in\{1,...,N\}$ and $P\in\oplus_{0\leq m\leq 2N-2}\cC^{1,0}_m\otimes\cC^{0,1}_{2N-2-m}$. Let $\rho:\cA\to[0,1]$ be a smooth plateau function supported in $K_2$ with $\rho\equiv1$ on $K_1$. For each $R>0$, we then set $\rho_R(\alpha):=\rho(R^{-1}\alpha)$, and note that $\rho_R$ is supported in $K_{2R}$. By Proposition \ref{prop:moments}, $\nu$ is a $\hat D^\omega_\infty G_\C$-covariant measure; differentiating \eqref{eq:quasi-invariance-bis} at $h=\ind_{G_\C}$ gives for all $b\in\kB,m\geq1$, and $P\in\cC$:
    \begin{equation}
    \int_\cA\bJ_{b,-m}(P)\rho_R\d\nu=\int_\cA\cJ_{b,-m}(\rho_R)P\d\nu.    
    \end{equation}
    Applying this to a polynomial $P$ satisfying the induction hypothesis, the term in the right-hand-side is bounded in absolute value by $\Vert P\Vert_{L^1(\nu)}\Vert\cJ_{b,-m}(\rho_R)\Vert_{L^\infty(\nu)}\leq\mathfrak c\nu(\cA\setminus K_R)$, for some constant $\mathfrak c>0$ independent of $R$. Hence, $\lim_{R\to\infty}\int_\cA\bJ_{b,-m}(P)\rho_R\d\nu=0$. Iterating the argument another time, we arrive at
    \[\lim_{R\to\infty}\int_\cA\bJ_{a,-n}(\bJ_{b,-m}(P))\rho_R\d\nu=0,\]
    for all $P\in\oplus_{0\leq k\leq2N-2}\cC^{1,0}_k\otimes\cC^{0,1}_{2N-2-k}$ and all $n,m\geq1$, $a,b\in\kB$. We get also the same conclusion for the polynomials $\bar\bJ_{a,-n}(\bJ_{b,-m}(P))$ and $\bar\bJ_{a,-n}(\bar\bJ_{b,-m}(P))$. Since $\oplus_{0\leq k\leq2N}\cC^{1,0}_k\otimes\cC^{0,1}_{2N-k}$ is spanned by such polynomials, we get in particular that $\lim_{R\to\infty}\int_\cA|P|^2\rho_R\d\nu$ exists and is finite for all $P\in\cC^{1,0}_N$. It then follows from Fatou's lemma that $|P|^2$ is integrable, ending the proof that $\cC^{1,0}_N\subset L^2(\nu)$. This completes the induction, namely $\cC^{1,0}\subset L^2(\nu)$.

     Next, we show that these moments are given by \eqref{eq:moments} by induction on $|\bk|+|\tilde\bk|$. The result holds for $|\bk|+|\tilde\bk|=0$ since $\nu$ is a probability measure. Let $N\geq1$ be such that the result for all $\bk,\tilde\bk$ with $|\bk|+|\tilde\bk|\leq N-1$. Let $1\leq m\leq N$ and $\bk,\tilde\bk\in\cT$ with $|\bk|+|\tilde\bk|= N-m$. From the proof that $\cC^{1,0}\subset L^2(\nu)$ right above, we have
    \[\int_\cA\bJ_{b,-m}(\Psi_{\bk,\tilde\bk})\d\nu=0=\int_\cA\bar\bJ_{b,-m}(\Psi_{\bk,\tilde\bk})\d\nu.\]
    Since all the states $\Psi_{\bk,\tilde\bk}$ with $|\bk|+|\tilde\bk|=N$ are obtained with this procedure, we deduce that $\int\Psi_{\bk,\tilde\bk}\d\nu=0$ for all $\bk,\tilde\bk\in\cT$ with $|\bk|+|\tilde\bk|=N$. This completes the induction and ends the proof of \eqref{eq:moments}.
\end{proof}

Recall from the proof of Proposition \ref{prop:existence} that $\cM\subset\P(\cA)$ denotes the set of accumulation points of the family $(\nu^\sigma_\kappa)_{\sigma>0}$ as $\sigma\to0$. For the statement of the next proposition, we make explicit the dependence on $\kappa$ and write $\cM_\kappa=\cM$.
\begin{proposition}\label{prop:uniqueness}
For all $\kappa<-2\check h$, any $\nu\in\cM_\kappa$ is characterised by its moments.
\end{proposition}

\begin{proof}
In this proof, we fix a $\kappa<-2\check h$ and a measure $\nu_\kappa\in\cM_\kappa$, and we will show that $\nu_\kappa$ is characterised by its moments. However, we will need to see the level as a variable in some parts of the proof (which will be denoted $\lambda$, with $\kappa$ remaining fixed), so we will make the dependence on the level explicit when relevant, e.g. we will write $\bL_u^\kappa=\bL_u=\cL_u+\check\kappa\ralpha(u)$. From the previous proposition, $\nu_\kappa$ has finite moments, and we denote by $\bar\cC^\kappa$ the closure of $\cC$ in $L^1(\nu_\kappa)$.

    Let $\tilde\cV_\kappa$ be the highest-weight representation generated by the representation $(\bL_u)_{u\in L^\omega\kg_\C}$ acting on $\ind$:
    \[\tilde\cV_\kappa:=\mathrm{span}\{\bL^{\circ\bk}\ind|\,\bk\in\cT\}\subset C^\infty(\cA),\]
    where we are using the notation introduced in Section \ref{subsec:hw} and \eqref{eq:def-big-circ}. Since $\bL_u$ is a first order differential operator in the coordinates $(\ralpha_{b,m},\bar\ralpha_{b,m})_{b\in\kB,m\geq1}$ and each coordinate $\ralpha_{b,m}$ is itself polynomial in $(\alpha_{a,n},\bar\alpha_{a,n})_{a\in\kB,n\geq1}$, we have $\tilde\cV_\kappa\subset\cC$, and we denote by $\cH_\kappa$ its closure in $L^2(\nu_\kappa)$.

By differentiating the $L^\omega_1G$-covariance formula \eqref{eq:quasi-invariance} $\chi=\ind_G$ (and using plateau functions as in the proof of Proposition \ref{lem:moments}), we obtain $\langle\bL_u^\kappa F,G\rangle_{\cH_\kappa}=-\langle F,\bL_{u^*}^\kappa G\rangle_{\cH_\kappa}$ for all $u\in L^\omega\kg_\C$ and all $F,G\in\tilde\cV_\kappa$. This implies that both $\bL_u^\kappa$ and its adjoint $(\bL_u^\kappa)^*$ on $\cH_\kappa$ are densely defined on $\cH_\kappa$, so they are both closable by \cite[Theorem VIII.1]{Reed-Simon_I} and their closures satisfy $(\bL_u^\kappa)^*=-\bL_{u^*}^\kappa$. Specialising to the canonical basis $b\otimes z^m$, we get $\bL_{b,m}^*=\bL_{b,-m}$ for all $b\in\kB,m\geq1$ (recall $\kB\subset i\kg$), which coincides with \cite[(14.130)]{CFTbook} and gives the unitarity of the representation $(\bL_u^\kappa)_{u\in L^\omega\kg_\C}$ acting on the Hilbert space $\cH_\kappa$ as unbounded operators defined on the common dense subspace $\tilde\cV_\kappa$. In the terminology of \cite[Section 1.4]{Goodman-Wallach}, $(\tilde\cV_\kappa,(\bL_u^\kappa)_{u\in L^\omega\kg_\C})$ is a standard module and \cite[Theorem 4.4]{Goodman-Wallach} states that the operator $\bL_u^\kappa+\bL_{u^*}^\kappa$ exponentiates to a one-parameter group of unitary operators on $\cH_\kappa$, for each $u\in\hat D^\omega_\infty\kg_\C$. In particular, we have a vector $w^\kappa_u:= e^{\bL_u^\kappa+\bL_{u^*}^\kappa}(\ind)\in\cH_\kappa$. Note that this implies $|w^\kappa_u|^2\in\bar\cC^\kappa$. 

Let $\lambda\in(\kappa,-2\check h)$ and fix $\nu_\lambda\in\cM_\lambda$. From the very definition of the measures $\nu^\sigma_\kappa$ and the positivity of $\sS$, we have $\bar\cC^\lambda\subset\bar\cC^\kappa$. Hence, we get a vector $|w_u^\lambda|^2\in\bar\cC^\kappa$. For all $\lambda\in(\kappa,-2\check h)$, we can write the vector $w_u^\lambda$ in the coordinate $\ralpha$ as the solution to an ODE, namely
\begin{equation}\label{eq:ODE}
w_u^\lambda(\ralpha)=e^{\lambda\int_0^1\ralpha_t(u)\d t}\qquad\text{with}\qquad\del_t\ralpha_t=K_{\ralpha_t}(u+u^*);\;\ralpha_0=\ralpha.
\end{equation}
Indeed, the dynamic is generated by the fundamental vector field associated to the infinitesimal left multiplication by $e^{tu+tu^*}$ as in \eqref{eq:vf_L}. This formula implies in particular that $w_u^\lambda$ is independent of the choice of measure in $\cM_\lambda$. Additionally, for all $\ralpha\in\cA$, the map $\lambda\mapsto w_u^\lambda(\ralpha)$ is analytic in the strip $\{\kappa<\Re(\lambda)<-2\check h\}$, and we have the identity $|w_u^\lambda(\ralpha)\overline{w^{\bar\lambda}_u(\ralpha)}|=|w_u^{\Re(\lambda)}(\ralpha)|^2$ in that strip. Consequently, setting $W^\lambda_u:=w_u^\lambda\overline{w_u^{\bar\lambda}}$, we have $W_u^\lambda\in L^1(\nu_\kappa)$ and the map $\lambda\mapsto W_u^\lambda$ is analytic with values in $L^1(\nu_\kappa)$. In particular, the whole map is determined by its values on the interval $(\kappa,-2\check h)$, which are themselves determined by the moments of $\nu_\kappa$ (since $W_u^\lambda\in\bar\cC^\kappa$ for $\kappa<\lambda<-2\check h$).

Finally, we consider $W_{\epsilon u}^{\lambda+i\epsilon^{-1}}$ for $\lambda\in(\kappa,-2\check h)$ fixed and study the asymptotic behaviour as $\epsilon\to0$. Coming back to \eqref{eq:ODE}, we have $w_{\epsilon u}^{\lambda+i\epsilon}(\ralpha)=e^{i\ralpha(u)+\epsilon\lambda\ralpha(u)+o(\epsilon)}$, so that $W_{\epsilon u}^{\lambda+i\epsilon^{-1}}(\ralpha)\to e^{2i\Re(\ralpha(u))}$, with the convergence being clearly uniform on each $K_R$, $R>0$. In fact, since the sequence is $(W_{\epsilon u}^{\lambda+i\epsilon^{-1}})_{\epsilon>0}$ is bounded in $L^1(\nu_\kappa)$, the convergence is easily promoted to a convergence in $L^1(\nu_\kappa)$. Now, observe that the linear span of the family $(\ralpha\mapsto e^{2i\Re(\ralpha(u))})_{u\in\hat D^\omega_\infty\kg_\C}$ forms a dense subspace of $L^1(\nu_\kappa)$, with the map $u\mapsto\int_\cA e^{2i\Re(\ralpha(u))}\d\nu(\ralpha)$ defining the Fourier transform (or characteristic function) of $\nu_\kappa$. Hence, for each $\lambda\in(\kappa,-2\check h)$, the linear span of the family $(W_{\epsilon u}^{\lambda+i\epsilon^{-1}})_{u\in\hat D^\omega_\infty\kg_\C,\epsilon>0}$ is itself dense in $L^1(\nu_\kappa)$. In particular, $\nu_\kappa$ is determined by the numbers $(\int_\cA W^{\lambda+it}_u\d\nu)_{u\in\hat D^\omega_\infty\kg_\C,t>0}$, which are themselves determined by the moments of $\nu_\kappa$ from the previous paragraph. This concludes the proof.
\end{proof}

    \subsection{Shapovalov forms}\label{subsec:shapo}

So far, we have proved that the unitarising measure $\nu_\kappa$ is characterised by either one of these four properties: the $L^\omega_1G$-covariance formula \eqref{eq:quasi-invariance}; the $\hat D^\omega_\infty G_\C$-covariance formula \eqref{eq:quasi-invariance-bis}; the moment computation \eqref{eq:moments}; the moment generating function is $\mathfrak F_\nu(\kappa,h)=1$ for all $h\in\hat D^\omega_\infty G_\C$. We will now give infinitesimal versions of \eqref{eq:quasi-invariance} and \eqref{eq:quasi-invariance-bis} (partly conjectural for the latter), which can be rephrased algebraically as an expression for the Shapovalov forms of the different Kac--Moody representations used in this paper.

        \subsubsection{Real representations}
First, we tie up some loose ends about the operators $(\cL_u)_{u\in L^\omega_0\kg_\C}$. The next statement is a straightforward adaptation of the beginning of the proof of Proposition \ref{prop:uniqueness}. As already mentioned there, this means that the representation is unitary. 
\begin{theorem}\label{thm:unitary}
    Let $\kappa<-2\check h$. For all $u\in L^\omega\kg_\C$, the operator $\bL_u$ is closable on $L^2(\nu_\kappa)$, and the closures satisfy $\bL_u^*=-\bL_{u^*}$.
\end{theorem}

The vector $w_u^\kappa=e^{\bL_u+\bL_{u^*}}(\ind)$ used in the proof has a special meaning. First, observe that $\Omega(e^u,\alpha)=e^{\check\kappa\sS}e^{\cL_u+\cL_u^*}e^{-\check\kappa\sS}$ since the operator $\cL_u+\cL_{u^*}$ is the infinitesimal generator of the left multiplication by $e^{tu}$. This almost coincides with $w_u^\lambda$ since $\bL_u=e^{\check\kappa\sS}\cL_ue^{-\check\kappa\sS}$ for $u\in\hat D^\omega_\infty\kg_\C$. Note however that $\bL_{u^*}=\cL_{u^*}\neq e^{\check\kappa\sS}\cL_{u^*}e^{-\check\kappa\sS}$. On the other hand, since the representation $(\bR_u)_{u\in L^\omega\kg_\C}$ commutes with the representation $(\bL_u)_{u\in L^\omega\kg_\C}$, we have $\bR_v(w_u^\kappa)=\cR_v(w_u^\kappa)=e^{\bL_u+\bL_{u^*}}(\cR_v\ind)=0$ for all $v\in\hat D^\omega\kg_\C$. In other words, the function $w_u^\kappa$ is holomorphic with respect to the \emph{left}-invariant complex structure. Hence, the covariance modulus admits a decomposition $\Omega(\chi,\alpha)=\Re(\Omega^\mathrm{hol}(\chi,\alpha))$ where $\alpha\mapsto\Omega^\mathrm{hol}(\chi,\alpha)$ is a left-holomorphic function on $\cA$ for each $\chi\in L^\omega_1G$. The careful reader will have noticed that we cannot a priori apply $\cR_v$ to functions on $\cA$ since these operators only act on functions on $\cA^\omega$. However, we can extend the values of $\Omega^\mathrm{hol}(\chi,\cdot)$ from $\cA^\omega$ to $\cA$ as in Proposition \ref{prop:Omega}. It is then very natural to introduce the following integral transform on $L^2(\nu_\kappa)$:
\begin{equation}\label{eq:def-Hecke}
\mathfrak H_F(\chi):=\int_\cA F(\alpha)e^{-\frac{\check\kappa}{2}\Omega^\mathrm{hol}(\chi,\alpha)}\d\nu_\kappa(\alpha),\qquad\forall F\in L^2(\nu_\kappa),\forall\chi\in L^\omega_1G,
\end{equation}
expressing the ``holomorphic part" of the left translation by $\chi$. In this way, the unitarity of the Goodman--Wallach representation gets nicely interpreted in the covariance property of $\nu_\kappa$. It would be interesting to investigate if $\mathfrak H_F$ also admits a holomorphic factorisation (in the variable $\chi$), and wether this is at all related to Hecke operators.\smallskip

The situation for the operators $(\cR_u)_{u\in L^\omega_0\kg_\C}$ is more subtle: since $\nu_\kappa$ doesn't charge the space of continuous loops, it is not even clear how to define $\cR_u$ on a dense subspace of $L^2(\nu_\kappa)$ since its definition involves the inverse loop. Moreover, even if this were the case, we would get formally $\cR_u^*=-\cR_{u^*}-\check\kappa\overline{\cR_u(\sS)}$, which doesn't coincide with $-\bR_{u^*}$ since $\sS\neq\hat\sS$.

\begin{conjecture} \label{conj:R}
    Sample $\gamma=g\hat g^{-1}$ from $\nu^\sigma_\kappa$ from Section \ref{subsec:existence} and denote by $\hat\nu_\kappa^\sigma$ the law of $\hat g^{-1}\del\hat g$. The family of measures $\hat\nu^\sigma_\kappa$ converges weakly as $\sigma\to0$ to a measure $\hat\nu_\kappa$ on $\hat\cA:=\{\alpha^*|\,\alpha\in\cA\}$. Moreover, the operators $(\bR_u)_{u\in L^\omega\kg_\C}$ are densely defined and closable on $L^2(\hat\nu_\kappa)$ and they satisfy the relations $\bR_u^*=-\bR_{u^*}$. 
\end{conjecture}

        \subsubsection{Complex representations}
It remains to discuss the Shapovalov form of the representations $(\bJ_u,\bar\bJ_u)_{u\in L^\omega\kg_\C}$. Answering this questions means finding a Hermitian form $\cQ$ on $L^2(\nu_\kappa)$ such that $\cQ(\bJ_uF,G)=\cQ(F,\bJ_{u^*}G)$ for all $u\in L^\omega\kg_\C$ and all $F,G$ in the domains of $\bJ_u$ and $\bJ_{u^*}$ respectively. Equivalently, we look for a self-adjoint operator $\Theta$ on $L^2(\nu_\kappa)$ such that $\Theta\circ\bJ_u^*=\bJ_u\circ\Theta$ in a suitable sense.

Let us define $L^2_\mathrm{hol}(\nu_\kappa)$ to be the closure of $\cV=\C[(\alpha_{b,m})_{b\in\kB,m\geq1}]$ in $L^2(\nu_\kappa)$, which is equipped with the representation $(\bJ_u)_{u\in L^\omega\kg_\C}$ (the other representation acting trivially in this space). Since $(L^2(\nu_\kappa),(\bJ_u,\bar\bJ_u)_{u\in L^\omega\kg_\C})$ is the tensor product of two copies of $(L^2_\mathrm{hol}(\nu_\kappa),(\bJ_u)_{u\in L^\omega\kg_\C})$, it is sufficient to restrict the discussion to $L^2_\mathrm{hol}(\nu_\kappa)$.

\begin{conjecture}\label{conj:shapo}
    Assume Conjecture \ref{conj:R} holds. There exists a set of full $\nu_\kappa$-measure $\cA^\omega\subset E\subset\cA$ and a holomorphic function $\Xi$ on $E$ such that $\Re(\Xi)=\hat\sS-\sS$ on $\cA^\omega$. 
\end{conjecture}

Arguing as in the proof of Proposition \ref{prop:existence}, it is not too hard to see that the sequence $(\hat\nu^\sigma_\kappa)_{\epsilon\to0}$ is tight as~$\epsilon\to0$, and we denote by $\hat\nu_\kappa$ any subsequential limit (which we conjecture is unique). Heuristically, we have the interpretation ``$\d\hat\nu_\kappa(\gamma)=e^{-\check\kappa\hat\sS(\gamma)}D\gamma$", so that formally $\frac{\d\hat\nu_\kappa}{\d\nu_\kappa}=|e^{-\frac{\check\kappa}2\Xi}|^2$. Using Proposition \ref{prop:km_real} and the definition of $\sS$, one easily checks that the function $\hat\sS-\sS$ is harmonic on $\cA^\omega$ (with respect to the bi-invariant metric, see Remark \ref{rem:killing}), hence $\hat\sS-\sS=\Re(\Xi)$ for some holomorphic function $\Xi$ on $\cA^\omega$. Therefore, the non-trivial part of the conjecture is the extension of $\Xi$ to a holomorphic function on a larger set of full $\nu_\kappa$-mass.

Assuming the validity of the conjecture, the formula 
\[\Theta F(\alpha):=F(\hat\alpha^*)e^{-\frac{\check\kappa}2\Xi(\alpha)}\]
defines a self-adjoint operator on $L^2(\nu_\kappa)$. In this definition, we are also assuming that the random variable $\hat\alpha$ with law $\hat\nu_\kappa$ is measurable with respect to $\alpha$. An elementary rephrasing of the integration by parts formula for the representation $(\bJ_u)_{u\in L^\omega\kg_\C}$ shows that the Hermitian form $\cQ(F,G):=\langle F,\Theta G\rangle_{L^2_\mathrm{hol}(\nu_\kappa)}$ is the Shapovalov form of the representation $(L^2_\mathrm{hol}(\nu_\kappa),(\bJ_u)_{u\in L^\omega\kg_\C})$.

\bibliographystyle{alpha}
\bibliography{loop-groups}

\end{document}